\documentclass[12pt]{article}
\usepackage[colorlinks=true,linkcolor=blue,citecolor=Red]{hyperref}
\usepackage{breakcites}

\usepackage{mathtools}

\usepackage{subfigure}

\usepackage{algorithm}
\usepackage{algpseudocode}

\usepackage{latexsym}
\usepackage{graphics}
\usepackage{amsmath}
\usepackage[algo2e]{algorithm2e}
\usepackage{xspace}
\usepackage{amssymb}
\usepackage{psfrag}
\usepackage{epsfig}
\usepackage{amsthm}
\usepackage{url}
\usepackage{pst-all}
\usepackage{bbm}

\usepackage[title]{appendix}

\usepackage{color}

\definecolor{Red}{rgb}{1,0,0}
\definecolor{Blue}{rgb}{0,0,1}
\definecolor{Olive}{rgb}{0.41,0.55,0.13}
\definecolor{Yarok}{rgb}{0,0.5,0}
\definecolor{Green}{rgb}{0,1,0}
\definecolor{MGreen}{rgb}{0,0.8,0}
\definecolor{DGreen}{rgb}{0,0.55,0}
\definecolor{Yellow}{rgb}{1,1,0}
\definecolor{Cyan}{rgb}{0,1,1}
\definecolor{Magenta}{rgb}{1,0,1}
\definecolor{Orange}{rgb}{1,.5,0}
\definecolor{Violet}{rgb}{.5,0,.5}
\definecolor{Purple}{rgb}{.75,0,.25}
\definecolor{Brown}{rgb}{.75,.5,.25}
\definecolor{Grey}{rgb}{.5,.5,.5}

\usepackage{bbm}
\newcommand{\ind}{\mathbbm{1}}

\newcommand{\R}{\mathbb{R}}

\newcommand{\N}{\mathbb{N}}
\newcommand{\ip}[2]{\langle{#1},{#2}\rangle} 
\renewcommand{\ip}[2]{\left\langle#1,#2\right\rangle}

\renewcommand{\R}{\mathbb{R}}

\newcommand{\bT}{\boldsymbol{\mathcal{T}}}
\newcommand{\bX}{\boldsymbol{X}}

\newcommand{\M}{\mathcal{M}}
\newcommand{\cN}{{\bf \mathcal{N}}}

\newcommand{\ignore}[1]{\relax}

\newtheorem{theorem}{Theorem}[section]
\newtheorem{remark}[theorem]{Remark}
\newtheorem{lemma}[theorem]{Lemma}

\newtheorem{proposition}[theorem]{Proposition}

\newtheorem{definition}[theorem]{Definition}

\renewcommand{\ip}[2]{\left\langle#1,#2\right\rangle}


\newcounter{parentnumber}

\makeatletter
\def\BState{\State\hskip-\ALG@thistlm}
\makeatother

\definecolor{Red}{rgb}{1,0,0}
\definecolor{Blue}{rgb}{0,0,1}
\definecolor{Olive}{rgb}{0.41,0.55,0.13}
\definecolor{Green}{rgb}{0,1,0}
\definecolor{MGreen}{rgb}{0,0.8,0}
\definecolor{DGreen}{rgb}{0,0.55,0}
\definecolor{Yellow}{rgb}{1,1,0}
\definecolor{Cyan}{rgb}{0,1,1}
\definecolor{Magenta}{rgb}{1,0,1}
\definecolor{Orange}{rgb}{1,.5,0}
\definecolor{Violet}{rgb}{.5,0,.5}
\definecolor{Purple}{rgb}{.75,0,.25}
\definecolor{Brown}{rgb}{.75,.5,.25}
\definecolor{Grey}{rgb}{.5,.5,.5}
\definecolor{Pink}{rgb}{1,0,1}
\definecolor{DBrown}{rgb}{.5,.34,.16}
\definecolor{Black}{rgb}{0,0,0}

\usepackage{float}

\usepackage{float}
\title{{\Large Information-Theoretic Guarantees for Recovering Low-Rank Tensors from Symmetric Rank-One Measurements}}
\author{{\sf Eren C. K{\i}z{\i}lda\u{g}}\thanks{Department of Statistics, University of Illinois Urbana-Champaign; e-mail: {\tt kizildag@illinois.edu}}}
\begin{document}
\maketitle
\begin{abstract}
In this paper, we investigate the sample complexity of recovering tensors with low symmetric rank from symmetric rank-one measurements. This setting is particularly motivated by the study of higher-order interactions and  the analysis of two-layer neural networks with polynomial activations (polynomial networks). Using a covering numbers argument, we analyze the performance of the symmetric rank minimization program and establish near-optimal sample complexity bounds when the underlying distribution is log-concave. Our measurement model involves random symmetric rank-one tensors, which lead to involved probability calculations. To address these challenges, we employ the Carbery-Wright inequality, a powerful tool for studying anti-concentration properties of random polynomials, and leverage orthogonal polynomials. Additionally, we provide a sample complexity lower bound based on Fano's inequality, and discuss broader implications of our results for two-layer polynomial networks.

\emph{Keywords:} Symmetric tensors, tensor recovery, rank minimization, covering numbers, low-rank, log-concave distributions.
\end{abstract}
\newpage
\tableofcontents
\newpage
\section{Introduction}
Tensors play a significant role in modern data science, as tensor-valued datasets frequently emerge in various applications such as neuroscience~\cite{beckmann2005tensorial}, imaging~\cite{zhou2013tensor,zhang2019tensor}, and signal processing~\cite{mesgarani2006content,nion2010tensor}. These applications often involve complex, multi-way interactions for which tensors serve as a natural representation. For a comprehensive overview, see~\cite{kolda2009tensor,bi2021tensors,auddy2024tensor} and references therein.

In this paper, we study the problem of recovering an unknown, order-$\ell$ tensor $\bT^*\in\R^{d\times \cdots \times d}$  from random measurements of the form:
\begin{equation}\label{eq:Measurement}
Y_i = \langle \bT^*, \mathcal{X}_i\rangle,\quad i=1,\dots,N.
\end{equation}
Here, an order-$\ell$ tensor $\bT^*\in\R^{d\times \cdots \times d}$ is a multidimensional array of numbers $\bT^*_{i_1,\dots,i_\ell}$, $i_1,\dots,i_\ell\in[d]:=\{1,\dots,d\}$, $N$ is the sample size,  $\mathcal{X}_i\in\R^{d\times \cdots \times d}$ are random, order-$\ell$ measurement tensors, and $\langle \cdot,\cdot \rangle$ is the Euclidean inner product in the ambient space $\R^{d^\ell}$.   Throughout, we assume that the order $\ell\in\mathbb{N}$ is fixed, while $d\to\infty$, reflecting the high-dimensional setting common in modern statistical applications. 

The measurement model~\eqref{eq:Measurement} is quite general and it encompasses various scenarios, such as  tensor completion, tensor regression, tensor estimation, and tensor PCA with appropriately chosen $\mathcal{X}_i$. See \cite{luo2023low} for a more detailed discussion. 

\paragraph{Symmetric Rank-One Measurements} Our particular focus is on symmetric, rank-one measurements, where $\mathcal{X}_i= \bX_i^{\otimes \ell}$ for independent and identically distributed (i.i.d.) random vectors $\bX_i\in\R^d$ with a distribution specified below. Here, $\otimes$ denotes the Segre outer product: for $\bX=(\bX(1),\dots,\bX(d))\in\R^d$ and any indices $i_1,\dots,i_\ell\in[d]$,
\[
(\bX^{\otimes \ell})_{i_1,\dots,i_\ell} := \bX(i_1)\cdots \bX(i_\ell).
\]
\paragraph{Distributional Assumption} We establish our results under the assumption that the entries of $\bX$ are i.i.d.\,samples of a log-concave distribution $\mathcal{D}$ on $\R$, i.e., $\bX\sim \mathcal{D}^{\otimes d}$.\footnote{Here, $\mathcal{D}^{\otimes d}$ represents the $d$-fold product measure $\mathcal{D}\otimes \cdots \otimes \mathcal{D}$.} A distribution $\mathcal{D}$ on $\R$ with density $f$ is log-concave if $-\log f$ is a convex function~\cite{klartag2005geometry,bagnoli2006log}. This is a very broad class that include a wide array of popular distributions such as  normal, uniform, exponential, Laplace, and Gamma, among many~\cite{bagnoli2006log}. Log-concave distributions are extensively studied in statistics and machine learning~\cite{samworth,walther2009inference,diakonikolas2017learning}, as well as in theoretical computer science~\cite{lovasz2007geometry}, combinatorics~\cite{stanley1989log}, and beyond.

\paragraph{Low Symmetric Rank Tensors} We assume that the unknown tensor $\bT^*$ has low symmetric rank: $\mathrm{rank}_S(\bT^*)\le r$ for some $r\in\mathbb{N}$, where the symmetric rank is defined as
\begin{equation}\label{eq:SymRank}
\mathrm{rank}_S(\bT^*):= \min\left\{r\ge 1:\bT^*=\textstyle\sum_{i\le r}\lambda_i \boldsymbol{v}_i^{\otimes \ell}, \lambda_1,\dots,\lambda_r\in\R,\boldsymbol{v}_1,\dots,\boldsymbol{v}_r\in\R^d\right\}. 
\end{equation}
Namely a tensor $\bT^*$ is of low symmetric rank iff it can be expressed as a sum of a small number of symmetric, rank-one components. See~\cite{comon2008symmetric, comon2020topology} for details on symmetric tensors.

\paragraph{Motivation} The setting we consider is particularly motivated by the study of interaction effects in statistics and the analysis of two-layer polynomial neural networks, as elaborated below.

One motivation for our setup arises from the study of higher-order interaction effects, where interactions may occur among pairs, triples, or more generally, $k$-tuples ($k\in\mathbb{N}$); see related discussions in \cite{bien2013lasso,basu2018iterative,hao2020sparse}. In such settings, it is often reasonable to assume that the unknown tensor exhibits a low-rank and/or sparse structure. This is particularly motivated by applications in biomedical fields and beyond; see \cite{sidiropoulos2012multi,hung2016detection} and \cite[Appendix~A]{hao2020sparse} for detailed discussions. A comprehensive analysis of cubic sketching, specifically focusing on pairwise and triple-wise interactions (the case $\ell=3$), was conducted in \cite{hao2020sparse}.

\subsection{Connections to Learning Two-Layer Polynomial Networks}\label{sec:NNs}
Our setup is closely related to the problem of learning two-layer neural networks with polynomial activations, commonly referred to as polynomial networks. Suppose that $\bX_i\in\R^d,i\in[N]$ are inputs and $Y_i$ are the corresponding
labels generated by a two-layer neural network of width $r$ and polynomial activation functions $\sigma(t)=t^\ell$:
\begin{equation}\label{eq:NN-Expression}
Y_i = \sum_{1\le j\le r} a_j^* \sigma\bigl(\langle \boldsymbol{W}_j^*,\bX_i\rangle) = \sum_{1\le j\le r}a_j^* \ip{\boldsymbol{W}_j^*}{\bX_i}^\ell, \quad \forall i\in[N].
\end{equation}
Given training data $(Y_i,\bX_i),i\in[N]$ the goal is to recover the weights  $a_i^*\in\R$ and $\boldsymbol{W}_i^*\in\R^d$ of the underlying network. This setup is known as the teacher-student model~\cite{goldt2019dynamics}. 

Although less common in practice, polynomial networks are still an active area of research. They possess strong expressive power and can simulate deep sigmoidal networks~\cite{livni2014computational}. Furthermore, they serve as a good approximation for networks with general non-linear activations and help in studying complex optimization landscapes~\cite{venturi2019spurious}. 
For further references on polynomial networks, see~\cite{soltanolkotabi2018theoretical,du2018power,emschwiller2020neural,sarao2020optimization,kizildag2022algorithms,martin2024impact,gamarnik2024stationary}.

We now return to~\eqref{eq:NN-Expression} and observe that the labels $Y_i$ satisfy
\begin{align}\label{eq:From-NN-to-Tensor}
Y_i &= \sum_{i_1,\dots,i_\ell\in[d]}\,\sum_{1\le j\le r}a_j^* \boldsymbol{W}_j(i_1)\cdots \boldsymbol{W}_j(i_r) \bX_i(i_1)\cdots \bX_i(i_\ell)=\langle \bT^*,\bX_i^{\otimes \ell}\rangle ,
\end{align}
where
\begin{equation}\label{eq:NN--TENSOR}
    \bT^* = \sum_{1\le j\le r}a_j^* \bigl(\boldsymbol{W}_j^*\bigr)^{\otimes \ell}
\end{equation}
is an order-$\ell$ symmetric tensor with symmetric rank at most $r$, $\mathrm{rank}_S(\bT^*) \le r$.  Therefore, given training data $(Y_i,\bX_i)$ generated by a teacher network, the problem of recovering the tensor representation $\bT^*$ of the underlying network fits precisely within the scope of the setting we consider. 
\begin{remark}\label{remark:tractability}
    After obtaining an estimate $\widehat{\bT}$, decomposing it as $\widehat{\bT} = \sum_{i\le s}\widehat{a}_i \widehat{\boldsymbol{W}}_i^{\otimes \ell}$, yields a set of weights for the underlying network. Although tensor decomposition is in general NP-hard~\cite{haastad1989tensor,haastad1990tensor,hillar2013most}, we can set aside the tractability issues and focus instead on the uniqueness perspective. Specifically, we ask: what is the smallest sample size required to uniquely identify the tensor representation of the underlying network? 
\end{remark}
\paragraph{Notation} An order-$\ell$ tensor $\bT\in\R^{d\times \cdots \times d}$ is a multidimensional array of numbers $\bT_{i_1,\dots,i_\ell}\in\R$, $i_1,\dots,i_\ell\in[d]$, where  $[d]:=\{1,\dots,d\}$ for any $d\in\mathbb{N}$. We denote the usual Euclidean inner product by $\langle \cdot,\cdot,\rangle$, where the underlying dimension will be clear from the context. Given an order-$\ell$ tensor $\bT\in\R^{d\times \cdots \times d}$, $\|\bT\|_F$ denotes its Frobenius norm $\sqrt{\langle \bT,\bT\rangle}$, with inner product taken in $\R^{d^\ell}$. Throughout, $\mathcal{D}$ denotes an arbitrary log-concave distribution on $\R$ with a density function. For $r\in\R,\exp(r):= e^r$. We denote the all ones vector by $\boldsymbol{1}$, and reflect the underlying dimension with a subscript. We use standard asymptotic notation, such as $\Omega(\cdot),O(\cdot),\Theta(\cdot),o(\cdot)$. We also use $\widetilde{\Omega}(\cdot),\widetilde{O}(\cdot),\widetilde{\Theta}(\cdot)$ for hiding logarithmic factors.
\section{Main Results}
Suppose that $r\in\mathbb{N}$ and $\bT^*\in\R^{d\times \cdots \times d}$ is an unknown, order-$\ell$ tensor with low symmetric rank, $\mathrm{rank}_S(\bT^*)\le r$ in the sense of~\eqref{eq:SymRank}. Given i.i.d.\,random vectors $\bX_i\sim \mathcal{D}^{\otimes d}$, the goal is to recover $\bT^*$ from symmetric, rank-one measurements $(Y_i,\bX_i)$, where
\begin{equation}\label{eq:DATA-COND}
Y_i = \bigl\langle \bT^*,\bX_i^{\otimes \ell}\rangle, \quad i\in[N].
\end{equation}
Motivated by a line of research in compressed sensing and low-rank matrix recovery~\cite{candes2005decoding,candes2006robust,cai2010singular,candes2010matrix,candes2012exact,eldar2012uniqueness,mu2014square}, we consider the following natural rank minimization program:
\begin{equation}\label{eq:Recovery-Program}
    \begin{aligned}
    &\underset{\bT\in\R^{d\times \cdots \times d}}{\min}
    & &\mathrm{rank}_S(\bT) \\
    &\text{subject to} 
    & &  \langle\bT,\boldsymbol{X}_i^{\otimes \ell}\rangle =Y_i,\,i=1,\dots,N
    \end{aligned}
\end{equation}
Throughout, $\widehat{\bT}$ denotes the solution to the program~\eqref{eq:Recovery-Program} given the data $(Y_i,\bX_i),i\in[N]$, as defined in~\eqref{eq:DATA-COND}.
Even though solving~\eqref{eq:Recovery-Program} is NP-hard~\cite{hillar2013most}, its analysis gives a benchmark for evaluating computationally efficient methods, see Section~\ref{sec:IT-BDS} for details. 

Our main result establishes a sample size upper bound for~\eqref{eq:Recovery-Program}. 
\begin{theorem}\label{thm:MAIN}
Let $C>2\ell^2$ be an arbitrary constant and $N\ge Crd$. Then, \eqref{eq:Recovery-Program} recovers all $\bT^*$ with $\mathrm{rank}_S(\bT^*)\le r$ with probability one.
\end{theorem}
See Section~\ref{sec:Road-to-Proof} for an outline of the proof. Several remarks are in order.

Since $\ell=O(1)$, Theorem~\ref{thm:MAIN} yields that $N=\Omega(rd)$ samples suffice for recovery. We highlight that Theorem~\ref{thm:MAIN} addresses the problem of \emph{strong recovery}: when $N=\Omega(dr)$,~\eqref{eq:Recovery-Program} successfully recovers \emph{any} $\bT^*$ with $\mathrm{rank}_S(\bT^*)\le r$. More formally,
\[
\mathbb{P}_{\bX_i\sim \mathcal{D}^{\otimes d}, i\in [N]}\Bigl[\forall \bT^*\in\bigl\{\bT\in\R^{d\times \cdots \times d}:\mathrm{rank}_S(\bT)\le r\bigr\}:\widehat{\bT} = \bT^*\Bigr]=1,
\]
provided $N=\Omega(dr)$. Theorem~\ref{thm:MAIN} holds under minimal assumptions. Notably, it applies to any arbitrary log-concave distribution $\mathcal{D}$. Furthermore, we impose no structural restrictions on $\bT^*$ beyond the rank constraint $\mathrm{rank}_S(\bT^*)\le r$. For a short discussion on the quadratic dependence of sample complexity on $\ell$, see Remark~\ref{remark:Quadratic}.

\paragraph{Implications for Neural Networks} Theorem~\ref{thm:MAIN}
has direct consequences for learning two-layer polynomial networks.
\begin{theorem}\label{coro:NN}
Suppose $S_r$ is the set of all two-layer neural networks $f:\R^d\to \R$ with activation function $\sigma(t)=t^\ell$ and width at most $r$: $f\in S_r$ iff there exists $a_1,\dots,a_r\in\R$ and $W_1,\dots,W_r\in\R^d$ such that $f(x) = \sum_{j\le r}a_j \sigma(\langle W_j,x\rangle)$ for all $x\in\R^d$. Let $\boldsymbol{f}^*\in S_r$ be fixed, $\bX_i\sim \mathcal{D}^{\otimes d},i\in[N]$ be i.i.d.\,random vectors and $Y_i = \boldsymbol{f}^*(\bX_i)$ for all $i\in[N]$. Provided $N=\Omega(dr)$, we have
\[
\mathbb{P}\Bigl[\forall \boldsymbol{f}^*\in S_r:\Bigl\{f\in S_r:Y_i = \boldsymbol{f}^*(\bX_i),\forall i\in[N]\Bigr\}=\{\boldsymbol{f}^*\}\Bigr]=1.
\]
\end{theorem}
See Section~\ref{sec:Road-to-Proof} for an outline of the proof. Theorem~\ref{coro:NN} asserts that $\boldsymbol{f}^*$ is the unique two-layer neural network of width at most $r$ fitting the training data, $Y_i = \boldsymbol{f}^*(\bX_i),\forall i\in[N]$. 

For polynomial networks, Theorem~\ref{coro:NN} yields superior sample complexity bounds in certain regimes. Notably, we impose no restrictions on the magnitude of components of $\bT^*$. When $\bT^*$ is interpreted as the tensor representation of a network~\eqref{eq:NN--TENSOR}, this means our sample bounds remain valid without constraints on the underlying weights, such as bounded norm. For further discussion, see Section~\ref{sec:NN-implication}.
\subsection{Comparison with Empirical Risk Minimization} Adopting a standard learning-theoretical lens, one may bypass~\eqref{eq:Recovery-Program} and instead attempt to recover $\bT^*$ from data~\eqref{eq:DATA-COND} by solving the empirical risk minimization (ERM) problem:
\[
\min_{\bT\in\mathcal{S}_\ell} \mathcal{L}_S(\bT),\quad\text{where}\quad \mathcal{L}_S(\bT) = \frac1N\sum_{i\le N}\left(Y_i - \bigl\langle \bT,\bX_i^{\otimes \ell}\bigr\rangle\right)^2,
\]
and $\mathcal{S}_\ell$ is the set of all order-$\ell$ symmetric tensors.\footnote{The constraint $\bT\in\mathcal{S}_\ell$ ensures that for any $1\le i_1,\dots,i_\ell \le d$, the entry $\bT_{i_1,\dots,i_\ell}$ is determined solely by the $d$-tuple $(\alpha_1,\dots,\alpha_d)$ where $\alpha_i$ counts the occurrences of $i\in[d]$ among $(i_1,\dots,i_\ell)$. This can be incorporated into the optimization by reducing the number of variables to $\binom{d+\ell-1}{\ell}$, which is the dimension of $\mathcal{S}_\ell$~\cite{comon2008symmetric}.} As we show next, this leads to a significantly worse sample complexity, since it inherently disregards the low-rank structure in $\bT^*$.
\begin{theorem}\label{thm:ERM-Sample}
Let $N^*(d,\ell):= \binom{d+\ell-1}{\ell}$ and $\bX_i\sim \mathcal{D}^{\otimes d}$, $i\in [N]$ be i.i.d.\,random vectors.
\begin{itemize}
    \item[(a)] Suppose that $N\ge N^*(d,\ell)$. Then, $\mathbb{P}\bigl[\bigl\{\bT:\mathcal{L}_S(\bT) = 0\bigr\} = \{\bT^*\}\bigr]=1$.
    \item[(b)] Suppose that $N<N^*(d,\ell)$. Then, $
    \mathbb{P}\bigl[\forall \bT^*,\forall M>0:Z_M(\bT^*)\ne \varnothing\bigr]=1$, where 
    \[
Z_M(\bT^*) =\Bigl\{\bT:\mathrm{rank}_S(\bT)<\infty,\mathcal{L}(\bT)=0,\mathbb{E}_{\bX\sim \mathcal{D}^{\otimes d}}\bigl[\bigl(\langle \bT,\bX^{\otimes \ell}\rangle - \langle \bT^*,\bX^{\otimes \ell}\rangle\bigr)^2\bigr]\ge M\Bigr\}.
\]
\end{itemize}
\end{theorem}
In $\mathrm{(b)}$, $\bX$ is a new sample drawn from $\mathcal{D}^{\otimes d}$ independent of $\bX_i$.
See Section~\ref{pf:ERM_SAMPLE} for the proof.  

While Part $\mathrm{(a)}$ establishes that the empirical risk has a unique global minimum for $N\ge N^*(d,\ell)$, Part $\mathrm{(b)}$ shows that for $N<N^*(d,\ell)$, it has global minima with arbitrarily large generalization error on the new data. This highlights a sharp contrast between structured and unstructured recovery: when $\ell=O(1)$ and $r$ is small, $N^*(d, \ell)$ is of order $\Theta(d^\ell)$, leading to a dramatically worse sample complexity than $\Theta(dr)$ bound per Theorem~\ref{thm:MAIN}. 
\subsection{Sample Complexity Lower Bounds}\label{sec:LOWER-BD}
In this section, we derive sample complexity lower bounds, showing that a bound of the form $\widetilde{\Omega}(dr^{1-\gamma})$, where $\gamma>0$ is arbitrarily small, is necessary. We focus on a discrete setting for simplicity, and leave the extension to the continuous case for future work. 

To obtain a lower bound, we impose a statistical model on the unknown tensor $\bT^*$ to be recovered. We establish the following proposition which of potential independent interest.
\begin{proposition}\label{prop:TENSOR_PACKING}
For every $\ell\ge 101$ and $r=o(d^{50}/\log^{50} d)$, there exists a set \[
\Psi \subset \{\bT\in\R^{d\times \cdots \times d}: \mathrm{rank}_S(\bT)\le r\}\] such that the following holds for all sufficiently large $d,r$.
\begin{itemize}
    \item $|\Psi|\ge e^{\Omega(dr^{0.98})}$.
    \item For any distinct $\bT,\bT'\in \Psi$, $\|\bT-\bT'\|_F \ge 1$.
\end{itemize}
Furthermore, for any $\bT\in \Psi$, $\|\bT\|_F\le r$ and that $d^{\ell/2}\bT$ has integer-valued entries.
\end{proposition}
Proposition~\ref{prop:TENSOR_PACKING} is essentially a packing numbers bound for symmetric tensors with low symmetric rank. Its proof is based on a variant of Gilbert-Varshamov lemma~\cite{gilbert1952comparison,varshamov1957estimate} from coding theory, which we derive  using the probabilistic method~\cite{alon2016probabilistic}. See Section~\ref{sec:Pf-Packing} for details.

By leveraging Proposition~\ref{prop:TENSOR_PACKING}, we derive the following sample complexity lower bound for estimating $\bT^*$.
\begin{theorem}\label{thm:FANO-CONV}
Suppose the assumptions on $\ell,d$ and $r$ from Proposition~\ref{prop:TENSOR_PACKING} hold. Let $B>0$ be a large constant, possibly depending on $d,r$. Assume that $\bT^*$ is drawn uniformly from  $\Psi$ in Proposition~\ref{prop:TENSOR_PACKING} and $\bX_i=(\bX_i(j):j\in[d])$, $i\in [N]$ are i.i.d.\,random vectors such that for each $i$, the entries $\bX_i(j)$ are i.i.d.\,samples from an arbitrary distribution on $[-B,B]\cap \mathbb{Z}$. Furthermore, suppose that $\bT^*$ and $\bX_i$ are independent. Consider the data $(Y_i,\bX_i)$ generated according to $Y_i = \ip{\bT^*}{\bX_i^{\otimes \ell}}$, $i\in[N]$. Then, for any $\delta>0$, 
    \[
\inf_{\widehat{\bT}}\,\mathbb{P}\bigl[\widehat{\bT}\ne \bT^*\bigr]\ge \delta \quad\text{if}\quad N = O\left(\frac{dr^{0.98}}{\log r+\ell \log(Bd)}\right),
    \]
    where the infimum is taken over all estimators $\widehat{\bT}$ of $\bT^*$ based on the data $(Y_i,\bX_i),i\in[N]$. 
\end{theorem}
See Section~\ref{pf:FANO} for the proof. Several remarks are in order.

Theorem~\ref{thm:FANO-CONV} applies to \emph{any} estimator $\widehat{\bT}$, whether deterministic or randomized. When $N=\widetilde{O}(dr^{0.98})$, every estimator incurs an estimation error of at least $\delta$, bounded away from zero. The specific choice of $0.98$ is arbitrary, a similar bound $N=\widetilde{O}(dr^{1-\gamma})$ holds for any $\gamma>0$ and sufficiently large $\ell$. This suggests that Theorem~\ref{thm:MAIN} is essentially tight up to $r^{o(1)}$ factors. 

We focus on the discrete setting here for simplicity, as the use of discrete variables enables cleaner statements and entropic arguments. We plan to extend our analysis to the continuous case in future work.

\subsection{Comparison with the Prior Work on Low-Rank Tensor/Matrix Recovery}
We now position our work within the broader context of low-rank matrix and tensor recovery and provide a brief comparison.

Low-rank models have been extensively studied, particularly within the framework of compressive sensing~\cite{candes2005decoding,candes2006robust,cai2010singular}. Our approach parallels the work~\cite{eldar2012uniqueness}, which focuses on recovering low-rank matrices,  and~\cite{mu2014square}, which addresses the recovery of low \emph{Tucker rank} tensors. Importantly though, both works adopt a measurement model in which $\mathcal{X}_i$ in~\eqref{eq:Measurement} consists of i.i.d.\,standard normal entries, simplifying probabilistic calculations. In contrast, our setting deals with symmetric, rank-one, log-concave measurements, which neither of these prior models cover. Moreover, their frameworks do not extend to two-layer polynomial networks as discussed earlier. Other related works~\cite{rauhut2017low,cai2020provable,ahmed2020tensor,grotheer2022iterative,luo2023low} study measurement models where $\mathcal{X}_i$ consists of i.i.d.\,sub-Gaussian entries, or, more generally, satisfies a certain tensor restricted isometry property (RIP). Whether symmetric, rank-one measurements in our setting satisfy tensor RIP remains an open question. (We note that while some of these works develop computationally efficient methods, our focus is solely on sample complexity.) 

Another line of research, dubbed as tensor completion, considers binary measurements where only a random subset of entries of a low-rank tensor is observed, see, e.g., ~\cite{ghadermarzy2018learning,yuan2017incoherent}. 

A particularly relevant work is~\cite{hao2020sparse}, which studies cubic sketching (the case $\ell=3$) and provides minimax-optimal convergence guarantees for the gradient descent.  Their analysis focuses on $\ell=3$ and a different setting where the vectors $\boldsymbol{v}_i\in\R^d$ in the decomposition of $\bT^*$ (see~\eqref{eq:SymRank}) are $s$-sparse, leading to a $\widetilde{\Omega}(rs\log d)$ sample bound.  Similarly,~\cite{shah2015optimal} considers separable measurements of the form $\mathcal{X}_i = a_i \otimes A_i$, where $a_i$ is a random vector and $A_i$ is a random matrix, obtaining a sample complexity bound of $\widetilde{\Omega}(dr)$.
\subsection{Comparison with Sample Complexity Bounds for Neural Networks}\label{sec:NN-implication} 
Sample complexity bounds for neural networks have been extensively studied in the literature; see, e.g., \cite{du2018power,vardi2022sample} and references therein. Focusing on two-layer polynomial networks discussed in  Section~\ref{sec:NNs}, Theorems~\ref{thm:MAIN} and~\ref{coro:NN} establish that $\Omega(dr)$ samples suffice to uniquely identify the tensor representation $\bT^*$ in~\eqref{eq:NN--TENSOR} of the underlying network (setting aside the tensor decomposition issue discussed in Remark~\ref{remark:tractability}). 

We now compare this with the existing bounds in the literature, particularly with the recent work~\cite{vardi2022sample}. Consider two-layer polynomial networks with activation $\sigma(z)=z^\ell$ and weight constraints $\|\boldsymbol{a}^*\|\le b$ and $\|\boldsymbol{W}^*\| \le B$, where $\boldsymbol{a}^* = (a_1^*,\dots,a_r^*)$ and $\boldsymbol{W}^* \in\R^{r\times d}$ with rows $\boldsymbol{W}_j^*\in\R^d$. Using a Rademacher complexity-based approach, \cite[Theorem~5]{vardi2022sample} establish a sample complexity bound of order $\widetilde{\Omega}(b^2 \cdot B^{2\ell}\cdot b_x^{2\ell})$ for inputs from $\{\boldsymbol{x}\in\R^d: \|\boldsymbol{x}\|_2 \le b_x\}$. 
Notably, if $\bX\sim\mathcal{D}^{\otimes d}$ with a sufficiently regular $\mathcal{D}$, then standard concentration arguments yield $\|\bX\| = \widetilde{O}(\sqrt{d})$ (w.h.p.). Consequently, the sample complexity bound from~\cite{vardi2022sample} for inputs drawn from $\mathcal{D}^{\otimes d}$ is of order $\widetilde{\Omega}\bigl(b^2 \cdot B^{2\ell}\cdot d^\ell\bigr)$. 

Let $N_s^* = \Omega(dr)$ be the sample bound in Theorems~\ref{thm:MAIN} and~\ref{coro:NN}. In the \emph{underparameterized} regime, specifically when $r=O(d^{\ell-1})$, we obtain $N_s^* < \widetilde{\Omega}\bigl(b^2\cdot B^{2\ell}\cdot d^\ell\bigr)$. We emphasize that this regime is relatively unexplored compared to its overparameterized counterpart, which has been extensively studied (see, e.g., \cite{bartlett2021deep}). Our approach demonstrates that low-rank tensors can offer valuable insights into the underparameterized setting.

In the \emph{overparameterized} regime, $r=\Omega(d^{\ell-1})$, our bound $N_s^*$ still remains competitive, particularly if the $a_i^*$ are of constant order (so that $\|\boldsymbol{a}\|=\Theta(\sqrt{r})$) or if the spectral norm $\|\boldsymbol{W}\|$ grows polynomially with $\max\{r,d\}$. %

Much of the prior literature on studying sample complexity of neural networks (such as~\cite{du2018power,golowich2020size,neyshabur2015norm}) relies on Rademacher complexity-based approaches. While these approaches avoid the dependence on network size (which, in some cases, may scale poorly), it does so at the expense of imposing constraints on the norm of underlying weights. In contrast, our tensor-based approach imposes no such restrictions, suggesting that tensor-based approaches may provide further insights into the theory of neural networks.
\subsection{Information-Theoretic Bounds and Noiseless Models}\label{sec:IT-BDS}
We now highlight the significance of information-theoretic bounds and noiseless models, as well as outline future directions.
\paragraph{Information-Theoretic Results} Information-theoretic guarantees such as ours serve as a foundational step towards computationally efficient methods. A vast body of literature in high-dimensional statistics focuses exclusively on such guarantees~\cite{wainwright2009information,wang2010information,tan2011rank,eldar2012uniqueness,zhang2013information,banks2018information,xu2018minimal}, as they offer benchmarks for polynomial-time algorithms. In our context, Theorems~\ref{thm:MAIN} and~\ref{thm:FANO-CONV} suggest that an algorithm requiring $\Theta(dr)$ measurements is sample-optimal, whereas one needing significantly more than $dr$ samples may be suboptimal and subject to improvement.
\paragraph{Noiseless Models}
Noiseless models like ours, where the linear measurements~\eqref{eq:DATA-COND} are observed exactly, are foundational in the literature on recovering low-rank structures. Numerous works focus solely on such models~\cite{wu2010renyi,keshavan2010matrix,riegler2015information,jalali2017universal,jalali2019toward,abbara2020universality}, as they provide clean theoretical benchmarks for identifying fundamental limits and guiding the development of computationally efficient methods.

\paragraph{Future Directions} Since the estimator~\eqref{eq:Recovery-Program} is computationally intractable, an important future direction is to analyze the sample complexity of computationally efficient methods, such as convex relaxations of~\eqref{eq:Recovery-Program}. Another avenue is extending our analysis to noisy models, where the measurements~\eqref{eq:DATA-COND} are corrupted with additive noise. A particularly relevant setting in this context is \emph{oblivious adversarial models}, which generalize matrix and tensor completion by allowing adversarially perturbed measurements. This framework provides a way to address robustness under structured noise conditions. For related work, see~\cite{candes2011robust,bhatia2017consistent,suggala2019adaptive,d2021consistent1,d2021consistent2,pensia2024robust}.
\section{The Road for Proving Theorems~\ref{thm:MAIN} and~\ref{coro:NN}}\label{sec:Road-to-Proof}
Both Theorem~\ref{thm:MAIN} and Theorem~\ref{coro:NN} are consequences of the following stronger result:
\begin{proposition}\label{Prop:MAIN}
    Let $C>2\ell^2$ be an arbitrary constant. Whenever $N\ge Crd$, we have
    \[
    \mathbb{P}\Bigl[\zeta_S(2r)\cap\bigl\{\bT\in\R^{d\times \cdots \times d}:\langle\bT, \bX_i^{\otimes \ell}\rangle=0,\forall i\in[N]\bigr\} = \{\boldsymbol{0}\}\Bigr]=1,
    \]
    where 
\begin{equation}\label{eq:zeta-set}
    \zeta_S(2r):= \Bigl\{\bT\in\R^{d\times \cdots \times d}:\mathrm{rank}_S(\bT)\le 2r, \|\bT\|_F=1\Bigr\}.
\end{equation}
\end{proposition}
We begin by showing why Proposition~\ref{Prop:MAIN} implies Theorems~\ref{thm:MAIN} and~\ref{coro:NN}.
Suppose that~\eqref{eq:Recovery-Program} has a solution different from $\bT^*$, and take any such $\widehat{\bT}\ne \bT^*$. Since $\bT^*$ automatically satisfies the constraints in~\eqref{eq:Recovery-Program} due to~\eqref{eq:DATA-COND} and $\mathrm{rank}_S(\bT^*)\le r$,  the optimal value of the objective function in~\eqref{eq:Recovery-Program} is at most $r$. Thus, $\mathrm{rank}_S(\widehat{\bT})\le r$. Since $\widehat{\bT}-\bT^*\ne \boldsymbol{0}$, we can define
\[
\bT_\Delta = \frac{\widehat{\bT}-\bT^*}{\|\widehat{\bT}-\bT^*\|_F},
\]
where $\|\bT_\Delta\|_F=1$. Furthermore, since $\mathrm{rank}_S$ is clearly subadditive, we have
\[
\mathrm{rank}_S(\bT_\Delta) = \mathrm{rank}_S(\widehat{\bT}-\bT^*) \le \mathrm{rank}_S(\widehat{\bT}) + \mathrm{rank}_S(\bT^*)\le 2r.
\]
Therefore, $\bT_\Delta\in \zeta_S(2r)$ for $\zeta_S(2r)$ arising in Proposition~\ref{Prop:MAIN}. Since $\widehat{\bT}$ is a solution of~\eqref{eq:Recovery-Program} we have $\langle \widehat{\bT},\bX_i^{\otimes \ell}\rangle = \langle \bT^*,\bX_i^{\otimes \ell}\rangle$; this yields  $\langle \bT_\Delta,\bX_i^{\otimes \ell}\rangle = 0$ for all $i\in[N]$. Consequently, if $\widehat{\bT}\ne \bT^*$ is a solution of~\eqref{eq:Recovery-Program}, we obtain
\[
\zeta_S(2r)\cap\bigl\{\bT\in\R^{d\times \cdots \times d}:\langle\bT, \bX_i^{\otimes \ell}\rangle=0,\forall i\in[N]\bigr\}\supseteq \{\boldsymbol{0},\bT_\Delta\}.
\]
Thus, Proposition~\ref{Prop:MAIN} implies Theorem~\ref{thm:MAIN}. Likewise, to see why Proposition~\ref{Prop:MAIN} implies Theorem~\ref{coro:NN}, suppose there exists an $f\in S_r$ such that $f\ne \boldsymbol{f}^*$ and $Y_i = f(\bX_i),\forall i\in[N]$. Writing $f(x) = \sum_{j\le r}a_j \sigma(\ip{W_j}{x})$ and $\bT = \sum_{j\le r}a_j W_j^{\otimes \ell}$ with $\bT\ne \bT^*$ for $\bT^*$ defined in~\eqref{eq:NN--TENSOR} and arguing analogously, we arrive at a contradiction to Proposition~\ref{Prop:MAIN}.

It thus suffices to prove Proposition~\ref{Prop:MAIN}. Building upon~\cite{eldar2012uniqueness,mu2014square} our approach is based on a \emph{covering number} argument. We say $\zeta'$ is an $\epsilon$-net for $\zeta_S(2r)$ if for any $\bT\in\zeta_S(2r)$, there exists a $\widehat{\bT}$ such that $\|\bT-\widehat{\bT}\|_F\le \epsilon$. The size of the smallest such $\epsilon$-net is known as the covering number of $\zeta_S(2r)$. For formal statements, see Definition~\ref{def:eps-net} and Section~\ref{SEC:COV}. We bound the covering numbers of $\zeta_S(2r)$ as follows. We identify in~\eqref{eq:Zeta-CP} a set of tensors, $\zeta_{\mathrm{CP}}(2r)\supset \zeta_S(2r)$ whose covering numbers is bounded by a recent result of~\cite{zhang2023covering} (reproduced in Theorem~\ref{thm:KZ23}). We then rely on a monotonicity property of covering numbers~\cite[Exercise~4.2.10]{vershynin2018high}: covering numbers of $\zeta_S(2r)$ is bounded above by that of $\zeta_{\mathrm{CP}}(2r)$.

Our approach requires controlling certain probabilistic terms. Since the data distributions in~\cite{eldar2012uniqueness} and~\cite{mu2014square} differ significantly from our setting, additional technical steps are necessary. Specifically, in their models, the measurement tensors $\mathcal{X}_j=(\mathcal{X}_j(i_1,\dots,i_\ell):i_1,\dots,i_\ell\in[d])$ consist of i.i.d.\,standard normal entries: $\mathcal{X}_i(i_1,\dots,i_\ell)\sim \cN(0,1)$ independently for each entry $i_1,\dots,i_\ell\in[d]$ and $j\in[N]$. Thus, the probability calculations are rather straightforward. 

In contrast, our measurement model, as defined in~\eqref{eq:DATA-COND}, involves random variables of form $\bX(i_1)\cdots \bX(i_\ell)$, $1\le i_1,\dots,i_\ell\le d$, where $\bX=(\bX(i):i\in[d]) \sim \mathcal{D}^{\otimes d}$ for an arbitrary log-concave distribution $\mathcal{D}$. These terms are degree-$\ell$ multivariate polynomials in the entries of $\bX$ and exhibit non-trivial dependencies. To address this challenge, we employ Carbery-Wright inequality~\cite{carbery2001distributional}, a powerful tool for studying the anti-concentration properties of low-degree polynomials of log-concave random variables, reproduced in Theorem~\ref{thm:CW}. Applying Carbery-Wright inequality, however, requires controlling the second moment of the polynomial. We achieve this by using orthogonal polynomial expansions w.r.t.\,the product measure $\mathcal{D}^{\otimes d}$. 

We provide a formal definition and useful properties of orthogonal polynomials in Section~\ref{SEC:Orthogonal}, following Steven Lalley's notes~\cite{LalleyNotes}. For further details, see~\cite{szego1939orthogonal}; see also~\cite{kunisky2019notes} for more on orthogonal polynomial expansions in statistical contexts, particularly regarding connections to computational hardness.
\section{Proof of Proposition~\ref{Prop:MAIN}}\label{sec:PF_PROP}
In this section, we present a proof of Proposition~\ref{Prop:MAIN}. To keep our exposition self-contained, we provide formal definitions whenever necessary.
\subsection{Covering Numbers for $\zeta_S(2r)$}\label{SEC:COV}
Our approach is based on a covering numbers argument for $\zeta_{\mathrm{S}}(2r)$, formally defined as follows.
\begin{definition}\label{def:eps-net}
    Let $(X,d)$ be a metric space and $\epsilon>0$. A subset $S_\epsilon\subseteq X$ is called an $\epsilon$-net for $X$ if for every $x\in X$, there exists $x'\in S_\epsilon$ such that $d(x,x')\le \epsilon$. The smallest size of an $\epsilon$-net (for $X$) is called the covering number of $X$, denoted by $\mathcal{N}(X,d,\epsilon)$.
\end{definition}
For more on covering numbers, see~\cite{vershynin2010introduction,vershynin2018high}. 

It appears challenging to bound $\mathcal{N}\bigl(\zeta_S(2r),\|\cdot\|_F,\epsilon\bigr)$ directly.  Instead, we bound the covering number of a related set of tensors that contains $\zeta_S(2r)$, and leverage a monotonicity property of covering numbers. To that end, we define the notion of CANDECOMP/PARAFAC (CP) rank~\cite{kolda2009tensor}.
\begin{definition}\label{def:CP-rank}
For an order-$\ell$ tensor $\bT\in\R^{d\times \cdots \times d}$, define its CP rank by
\[
\mathrm{rank}_{\mathrm{CP}}(\bT) = \min\left\{r\ge 1:\bT = \sum_{1\le j\le r}\boldsymbol{v}_1^{(j)}\otimes \cdots \otimes \boldsymbol{v}_\ell^{(j)}, \boldsymbol{v}_i^{(j)}\in\R^d,i\in[\ell],j\in[r]\right\}.
\]
Furthermore, let
\begin{equation}\label{eq:Zeta-CP}
    \zeta_{\mathrm{CP}}(r) = \Bigl\{\bT\in\R^{d\times \cdots \times d}:\mathrm{rank}_{\mathrm{CP}}(\bT)\le r, \|\bT\|_F=1\Bigr\}.
\end{equation}
\end{definition}
In light of~\eqref{eq:SymRank} and Definition~\ref{def:CP-rank}, it follows that $\mathrm{rank}_{\mathrm{CP}}(\bT)\le \mathrm{rank}_S(\bT)$ for any tensor $\bT$. Consequently, $ \zeta_{\mathrm{S}}(2r)\subset  \zeta_{\mathrm{CP}}(2r)$. 

We next record a useful monotonicity property, which allows us to bound the covering number of $\zeta_S(2r)$ by that of $\zeta_{\mathrm{CP}}(2r)$.
    \begin{lemma}{\cite[Exercise~4.2.10]{vershynin2018high}}\label{lemma:COV_Mono}
    For any $\epsilon>0$, 
    \[
    L\subset K \implies\mathcal{N}(L,d,\epsilon)\le \mathcal{N}(K,d,\epsilon/2).
    \]
\end{lemma}
Therefore, it suffices to control $\mathcal{N}\bigl(\zeta_{\mathrm{CP}}(r),\|\cdot\|_F,\epsilon\bigr)$. Such a bound on the covering numbers of tensors with low CP rank has only recently appeared in the literature and is provided below.
\begin{theorem}{\cite[Theorem~3.1]{zhang2023covering}}\label{thm:KZ23}
    Let $C>0$ be an absolute constant. For $d\ge 2$, $r>0$ and $\epsilon\in(0,2]$,
    \[
\log\mathcal{N}\Bigl(\zeta_{\mathrm{CP}}(r),\|\cdot\|_F,\epsilon\Bigr)\le r\ell d \log \frac1\epsilon + Cr\ell^2 d \log d.
    \]
\end{theorem}
Combining Lemma~\ref{lemma:COV_Mono} and Theorem~\ref{thm:KZ23} with $K=\zeta_{\mathrm{CP}}(2r)$, $L=\zeta_{\mathrm{S}}(r)$ and $d=\|\cdot\|_F$, we obtain:
\begin{lemma}\label{lemma:COV-S}
   Let $C>0$ be an absolute constant. For any $d\ge 2$, $r>0$ and $\epsilon\in(0,2]$,
    \[
\log\mathcal{N}\Bigl(\zeta_{S}(2r),\|\cdot\|_F,\epsilon\Bigr) \le 2r\ell d \log \frac{2}{\epsilon} + Cr\ell^2 d\log d.
    \]
\end{lemma}
\subsection{Probability Estimates }\label{prob:EST}
Our probability estimates crucially rely on Carbery-Wright inequality, a powerful tool for establishing anti-concentration bounds for polynomials of random vectors.
\begin{theorem}\label{thm:CW}{\cite[Theorem~8]{carbery2001distributional}}
    Let $\bX\in\R^d$ has a log-concave density and $P:\R^d\to\R$ be a polynomial with $\mathrm{deg}(P)=\ell$. Then, there exists an absolute constant $C>0$ such that for any $\epsilon>0$, 
    \[
\mathbb{P}\bigl[|P(\bX)|\le \epsilon\bigr]\le C\ell \left(\frac{\epsilon}{\sqrt{\mathbb{E}[P(\bX)^2]}}\right)^{\frac1\ell}.
    \]
\end{theorem}
We apply Theorem~\ref{thm:CW} with $\bX\sim \mathcal{D}^{\otimes d}$,  where $\mathcal{D}$ is log-concave, to the polynomial
\[
P(\bX) = \langle \bT,\bX^{\otimes \ell}\rangle,\quad\text{where}\quad \bT\in\zeta_S(2r).
\]
For this, we need to control $\mathbb{E}[P(\bX)^2]$ in the denominator, which is the subject of our next result.
\begin{proposition}\label{prop:f-Expand}
Let $\bT\in\R^{d\times \cdots \times d}$ be an order-$\ell$ tensor with bounded $\mathrm{rank}_S(\bT)$. Then, 
\[
\mathbb{E}_{\bX\sim \mathcal{D}^{\otimes d}}\Bigl[\langle \bT,\bX^{\otimes \ell}\rangle^2\Bigr]\ge \Xi \cdot \|\bT\|_F^2,
\]
where $\Xi = \mathcal{C}(\mathcal{D},\ell)^d$ for a finite constant $\mathcal{C}(\mathcal{D},\ell)>0$ depending only on $\ell$ and the moments of the distribution $\mathcal{D}$.
\end{proposition}
We prove Proposition~\ref{prop:f-Expand} using orthogonal polynomial expansion w.r.t.\,the product measure $\mathcal{D}^{\otimes d}$. See Section~\ref{SEC:Orthogonal} for a background and useful properties of orthogonal polynomials, as well as the complete proof of Proposition~\ref{prop:f-Expand}.
\subsection{Putting Everything Together: Proof of Proposition~\ref{Prop:MAIN}}
Equipped with Lemma~\ref{lemma:COV-S}, Theorem~\ref{thm:CW} and Proposition~\ref{prop:f-Expand}, we are ready to prove Proposition~\ref{Prop:MAIN}. We follow an outline similar to~\cite{eldar2012uniqueness,mu2014square}. 
\begin{proof}[Proof of Proposition~\ref{Prop:MAIN}] 
Let $\zeta'$ be an $\epsilon$-net for $\zeta_S(2r)$ with respect to $\|\cdot\|_F$ with the smallest cardinality. In light of Definition~\ref{def:eps-net} and Lemma~\ref{lemma:COV-S}, we obtain 
\begin{equation}\label{eq:SIZE-zeta'}
    |\zeta'|\le \exp\left(2r\ell d \log\frac{2}{\epsilon} + Cr\ell^2 d\log d\right).
\end{equation}
Next, take any $\widehat{\bT}\in \zeta'$ and consider $P(\bX) = \langle \widehat{\bT},\bX^{\otimes \ell}\rangle$, so that $\mathrm{deg}(P)=\ell$. Applying Proposition~\ref{prop:f-Expand} together with the fact $\|\widehat{\bT}\|_F=1$,  we obtain $\mathbb{E}[P(\bX)^2]\ge \Xi$. So, 
\begin{equation}\label{eq:MAX_PROB}
    \mathbb{P}\left[\max_{1\le i\le N}\bigl|\langle \widehat{\bT},\bX_i^{\otimes \ell}\rangle\bigr|<2\epsilon \log\frac1\epsilon\right]\le \left(\frac{C\ell}{\Xi^{1/\ell}}\left(2\epsilon\log \frac1\epsilon\right)^{\frac1\ell}\right)^N,
\end{equation}
using Carbery-Wright inequality, Theorem~\ref{thm:CW}, and the fact that $\bX_i,i\in[N]$ are i.i.d. 

Next, fix a $\bT\in \zeta_S(2r)$ and let $\widehat{\bT}\in\zeta'$ be such that $\|\bT-\widehat{\bT}\|_F\le \epsilon$. We have
\begin{align}
    \Bigl|\bigl|\langle \bT,\bX_i^{\otimes \ell}\rangle\bigr|-\bigl|\langle \widehat{\bT},\bX_i^{\otimes \ell}\rangle\bigr|\Bigr|&\le \Bigl|\bigl\langle \bT-\widehat{\bT},\bX_i^{\otimes \ell}\bigr\rangle\Bigr|\label{eq:reverse-triangle} \\
    &\le \|\bT-\widehat{\bT}\|_F \cdot \max_i \|\bX_i^{\otimes \ell}\|_F\label{eq:Cauchy-Schwarz}\\
    &\le \epsilon\cdot \max_i \|\bX_i\|_2^\ell\label{eq:ALGEBRA},
\end{align}
where~\eqref{eq:reverse-triangle} uses the reverse triangle inequality, $||x|-|y||\le |x-y|$ valid for all $x,y$, \eqref{eq:Cauchy-Schwarz} uses Cauchy-Schwarz inequality, and finally, \eqref{eq:ALGEBRA} uses the fact $\|\bT-\widehat{\bT}\|_F\le \epsilon$ and that $\|\bX_i^{\otimes \ell}\|_F = \sqrt{\sum_{i_1,\dots,i_\ell}\bX_i(i_1)^2\cdots \bX_i(i_\ell)^2} = \|\bX_i\|_2^{\ell}$. Using~\eqref{eq:ALGEBRA},
\begin{align}
    \max_{1\le i\le N}\bigl|\langle \bT,\bX_i^{\otimes \ell}\rangle\bigr|&\ge \max_{1\le i\le N} \bigl|\langle \widehat{\bT},\bX_i^{\otimes \ell}\rangle\bigr| -\epsilon \cdot \max_i \|\bX_i\|_2^\ell \nonumber\\
    &\ge \min_{\widehat{\bT}\in\zeta'}\max_{1\le i\le N} \bigl|\langle \widehat{\bT},\bX_i^{\otimes \ell}\rangle\bigr| -\epsilon \cdot \max_i \|\bX_i\|_2^\ell\label{eq:ALGEBRA-2}.
\end{align}
Notice that the right hand side of~\eqref{eq:ALGEBRA-2} is independent of $\bT$. Taking an infimum over $\bT\in \zeta_S(2r)$, we thus obtain
\begin{equation*}
 \inf_{\bT\in \zeta_S(2r)}   \max_{1\le i\le N}\bigl|\langle \bT,\bX_i^{\otimes \ell}\rangle\bigr|\ge \min_{\widehat{\bT}\in\zeta'}\max_{1\le i\le N} \bigl|\langle \widehat{\bT},\bX_i^{\otimes \ell}\rangle\bigr| -\epsilon \cdot \max_i \|\bX_i\|_2^\ell.
\end{equation*}
This yields the following inclusion property:
\begin{align}
    &\left\{ \inf_{\bT\in \zeta_S(2r)}   \max_{1\le i\le N}\bigl|\langle \bT,\bX_i^{\otimes \ell}\rangle\bigr|<\epsilon\log\frac1\epsilon\right\} \nonumber\\ 
    &\subset \left\{\min_{\widehat{\bT}\in\zeta'}\max_{1\le i\le N} \bigl|\langle \widehat{\bT},\bX_i^{\otimes \ell}\rangle\bigr|<2\epsilon\log\frac1\epsilon\right\}\cup \left\{\max_i \|\bX_i\|_2>\log^{1/\ell}\frac1\epsilon\right\}.\label{eq:event-include}
\end{align}
Thus,
\begin{align}
&\mathbb{P}\left[\inf_{\bT\in \zeta_S(2r)}   \max_{1\le i\le N}\bigl|\langle \bT,\bX_i^{\otimes \ell}\rangle\bigr|<\epsilon\log\frac1\epsilon\right]\nonumber \\
    &\le \mathbb{P}\left[\min_{\widehat{\bT}\in\zeta'}\max_{1\le i\le N} \bigl|\langle \widehat{\bT},\bX_i^{\otimes \ell}\rangle\bigr|<2\epsilon\log\frac1\epsilon\right] + \mathbb{P}\left[\max_i \|\bX_i\|_2>\log^{1/\ell}\frac1\epsilon\right]\label{eq:U--BD} \\
    &\le |\zeta'|\mathbb{P}\left[\max_{1\le i\le N} \bigl|\langle \widehat{\bT},\bX_i^{\otimes \ell}\rangle\bigr|<2\epsilon\log\frac1\epsilon\right] + \mathbb{P}\left[\max_i \|\bX_i\|_2>\log^{1/\ell}\frac1\epsilon\right]\label{eq:U--bd-2} \\
    &\le \exp\left(2r\ell d\log\frac2\epsilon + Cr\ell^2 d\log d-\frac{N}{\ell}\log\frac{1}{2\epsilon}+\frac{N}{\ell}\log \log\frac1\epsilon +N \log \frac{C\ell}{\Xi^{1/\ell}}\right) \label{EQ:BIG-EXPONENT}\\
    &+\mathbb{P}\left[\max_i \|\bX_i\|_2>\log^{1/\ell}\frac1\epsilon\right]\label{eq:COV-LEMMA-PROB-EST},
\end{align}
where~\eqref{eq:U--BD} follows by using~\eqref{eq:event-include} and taking a union bound,~\eqref{eq:U--bd-2} follows by taking another union bound over $\widehat{\bT}\in\zeta'$, and \eqref{EQ:BIG-EXPONENT} follows by combining~\eqref{eq:SIZE-zeta'} and~\eqref{eq:MAX_PROB}. Now, suppose that $N = Crd$ where $C>2\ell^2$. Then,
\[
2r\ell d \log \frac1\epsilon - \frac{N}{\ell}\log\frac1\epsilon = C'rd\log\epsilon,
\]
where $C'=\frac{C}{\ell}-2\ell>0$. With this, the term appearing in~\eqref{EQ:BIG-EXPONENT} is upper bounded by
\begin{equation}\label{eq:BIG-EXP-UP-BD}
    \epsilon^{C'\ell rd}\left(\log\frac1\epsilon\right)^{\frac{N}{\ell}}\cdot O_\epsilon(1),
\end{equation}
where the terms $O_\epsilon(1)$ depend only on $r,\ell,d$ and remain constant as $\epsilon\to 0$. Since $C'>0$, \eqref{eq:BIG-EXP-UP-BD} tends to zero as $\epsilon\to 0$. Thus, taking a countable sequence $(\epsilon_n)_{n\in\mathbb{N}}$ with $\epsilon_n\to 0$ and using the continuity of probability measures, we obtain
\[
\mathbb{P}\left[\inf_{\bT\in \zeta_S(2r)}   \max_{1\le i\le N}\bigl|\langle \bT,\bX_i^{\otimes \ell}\rangle\bigr|=0\right]=0.
\]
This concludes the proof of Proposition~\ref{Prop:MAIN}.
\end{proof}
\begin{remark}\label{remark:Quadratic}
    We remark on the quadratic dependence of $N$ on $\ell$ in Theorem~\ref{thm:MAIN} and Proposition~\ref{prop:f-Expand}. In proving Proposition~\ref{prop:f-Expand}, two terms play a key role: the covering upper bound~\eqref{eq:SIZE-zeta'} and the probability estimate~\eqref{eq:MAX_PROB}. For our argument to hold, the term $\frac{N}{\ell}\log\frac{1}{\epsilon}$ from~\eqref{eq:MAX_PROB} must dominate the term $2r\ell d\log\frac{1}{\epsilon}$ from~\eqref{eq:SIZE-zeta'}. This leads to a quadratic dependence on $\ell$. The $\frac1\ell$ scaling in $\frac{N}{\ell}\log\frac1\epsilon$ arises from~\eqref{eq:MAX_PROB} and appears intrinsic to the Carbery-Wright bound (e.g., for a polynomial $P(x)=x^\ell$ and $Z\sim \cN(0,1)$, $\mathbb{P}[|P(Z)|\le \epsilon]$ indeed scales like $O_\epsilon(\epsilon^{1/\ell})$). For the covering bound, we are unsure if the dependence of~\eqref{eq:SIZE-zeta'} and Lemma~\ref{lemma:COV-S} on $\ell$ is optimal, which we leave as an interesting question. (Thanks to an anonymous reviewer for raising this question.) 
\end{remark}

\section{Background on Orthogonal Polynomials}\label{SEC:Orthogonal}
Orthogonal polynomials play an important role in our setup.  In this section, we formally define orthogonal polynomials with respect to $\mathcal{D}$ and collect useful properties. The content here is largely adapted from  S.\,Lalley's notes~\cite{LalleyNotes}, see also the classical textbook by~\cite{szego1939orthogonal}.

Let $\mathcal{D}$ be a log-concave distribution on $\R$, and $\bX\sim \mathcal{D}^{\otimes d}$. That is, $\bX=(\bX(1),\dots,\bX(d))$, where $\bX(i)\sim \mathcal{D}$ are i.i.d. A family $\{P_n(x)\}_{n\ge 0}$ of polynomials are orthonormal w.r.t.\,$\mathcal{D}$ if
\[
\mathbb{E}_{X\sim \mathcal{D}}\bigl[P_n(X)P_m(X)\bigr] = \ind\{n=m\}.
\]
Given $\mathcal{D}$, such polynomials exist due to classical results of~\cite{szego1939orthogonal}. Building on~\cite{LalleyNotes} (and using its notation), the following proposition collects the necessary properties of orthonormal polynomial that we utilize.
\begin{proposition}\label{Prop:Orthonormal}
Let $\mathcal{D}$ be a log-concave distribution and $ 
m_i:=\mathbb{E}_{X\sim \mathcal{D}}[X_i]$ for $i\ge 0$.
\begin{itemize}
    \item[(a)] We have $m_i<\infty$ for all $i$. Furthermore, the associated family $\{P_n(x)\}_{n\ge 0}$ of orthogonal polynomials are given by: $P_0(x)=1$, and for $n\ge 1$, $P_n(x) = D_n(x)/\sqrt{D_n D_{n-1}}$ where
    \[
D_n(x) = 
\begin{vmatrix}
m_0 & m_1 & m_2 & \cdots & m_n \\
m_1 & m_2 & m_3 & \cdots & m_{n+1} \\
\vdots & \vdots & \vdots & \ddots & \vdots \\
m_{n-1} & m_n & m_{n+1} & \cdots & m_{2n-1} \\
1 & x & x^2 & \cdots & x^n
\end{vmatrix}
\]
and 
\[
D_n = \begin{vmatrix}
m_0 & m_1 & m_2 & \cdots & m_{n} \\
m_1 & m_2 & m_3 & \cdots & m_{n+1} \\
\vdots & \vdots & \vdots & \ddots & \vdots \\
m_n & m_{n+1} & m_{n+2} & \cdots & m_{2n}
\end{vmatrix}
\]
and that $0<D_n<\infty$ for all $n$.
\item[(b)] For any $n\ge 1$ and any $k<n$, we have
\[
\mathbb{E}_{X\sim \mathcal{D}}[X^n P_n(X)] = \sqrt{\frac{D_n}{D_{n-1}}}\quad\text{and}\quad \mathbb{E}_{X\sim \mathcal{D}}[X^k P_n(X)]=0.
\]
\item[(c)] For  $\boldsymbol{\alpha}=(\alpha_1,\dots,\alpha_d)\in\mathbb{N}_0^d$ and $\boldsymbol{X}=(X_1,\dots,X_d)\in\R^d$, the family $\{\boldsymbol{P_\alpha}\}$ of polynomials defined by
\[
\boldsymbol{P_\alpha}(\boldsymbol{X}) =\prod_{1\le i\le d}P_{\alpha_i}(X_i)
\]
are orthonormal with respect to the product measure $\mathcal{D}^{\otimes d}$.
\end{itemize}
\end{proposition}
\begin{proof}[Proof of Proposition~\ref{Prop:Orthonormal}]
    Part $\mathrm{(a)}$ is reproduced verbatim from~\cite[Proposition~1]{LalleyNotes}. The fact $m_i<\infty$ for log-concave $\mathcal{D}$ is well-known, see, e.g., \cite{bobkov2015concentration}. 
    
    We next show that $0<D_n<\infty$.  As noted, $m_i<\infty$, so $D_n<\infty$ for all $n$. Next, we prove $D_n>0$. Setting $\Upsilon_n =(X^i:0\le i\le n)\in\R^{n+1}$ as a column vector, it suffices to prove that $\mathbb{E}_{X\sim \mathcal{D}}[\Upsilon_n \Upsilon_n^T]$ is positive definite as $D_n:=\mathrm{det}(\mathbb{E}_{X\sim \mathcal{D}}[\Upsilon_n \Upsilon_n^T])$. Suppose that for some $\boldsymbol{v}=(v_0,\dots,v_n)\in\R^{n+1}$, we have $\boldsymbol{v}^T \mathbb{E}[\Upsilon_n \Upsilon_n^T]\boldsymbol{v} = \mathbb{E}[\ip{\Upsilon_n}{\boldsymbol{v}}^2]=0$. Then, $\ip{\Upsilon_n}{\boldsymbol{v}} = \sum_{0\le i\le n}v_i X^i=0$ almost surely w.r.t.\,$\mathcal{D}$. Hence, for the polynomial $Q(x) = \sum_{0\le i\le n}v_ix^i$, we have $Q(X)=0$. By the standard results of~\cite{caron2005zero}, we have  $Q(x)\equiv 0$ identically, i.e., $\boldsymbol{v} = \boldsymbol{0}$. So, $\boldsymbol{v}^T \mathbb{E}_{X\sim \mathcal{D}}[\Upsilon_n \Upsilon_n^T]\boldsymbol{v}=0$ iff $\boldsymbol{v}=0$, forcing $D_n=\mathrm{det}(\mathbb{E}_{X\sim \mathcal{D}}[\Upsilon_n \Upsilon_n^T])>0$.

    As for Part $\mathrm{(b)}$, the proof of Proposition 1 in~\cite{LalleyNotes} reveals 
    \[
    \mathbb{E}_{X\sim \mathcal{D}}[X^k D_n(X)] = \ind\{k=n\} D_n,
    \]
    for every $k\le n$. As $P_n(X)=D_n(x)/\sqrt{D_n D_{n-1}}$, per Part $\mathrm{(a)}$, Part $\mathrm{(b)}$ follows. 
    
    As for Part $\mathrm{(c)}$, fix $\boldsymbol{\alpha},\boldsymbol{\alpha'}\in\mathbb{N}_0^d$ and let $\bX \sim \mathcal{D}^{\otimes d}$ has i.i.d.\,entries. Then, 
    \begin{align*}
            \mathbb{E}_{\bX\sim \mathcal{D}^{\otimes d}}\left[\boldsymbol{P_\alpha}(\bX)\boldsymbol{P_{\alpha'}}(\bX)\right] &= \prod_{1\le i\le d}\mathbb{E}_{X_i\sim \mathcal{D}}\left[P_{\alpha_i}(X_i)P_{\alpha_i'}(X_i)\right] \\
            &= \prod_{1\le i\le d}\ind\{\alpha_i =\alpha_i'\} \\
            &= \ind\left\{\boldsymbol{\alpha}=\boldsymbol{\alpha'}\right\}.
        \end{align*}
        Thus, the family $\{\boldsymbol{P}_{\alpha}\}$ is indeed orthonormal.
\end{proof}
The following result is folklore and can be proven, e.g., by modifying the argument of~\cite{MSC-Post} together with an induction  over $d$ which we skip for simplicity.
\begin{proposition}\label{prop:BASIS}
    The family $\boldsymbol{P_\alpha}$ form an orthonormal basis for $L^2(\mathcal{D}^{\otimes d})$.
\end{proposition}
Proposition~\ref{prop:BASIS} asserts that any $f:\in\R^d\to \R$ with $\mathbb{E}_{\bX\sim \mathcal{D}^{\otimes d}}[f(\bX)^2]<\infty$ admits an expansion $f(\bX) = \sum_{\boldsymbol{\alpha}\in \mathbb{N}_0^d}\Theta_{\boldsymbol{\alpha}} \boldsymbol{P_\alpha}(\bX)$ such that $\mathbb{E}_{\bX\sim \mathcal{D}^{\otimes d}}[f(\bX)^2] =\sum_{\boldsymbol{\alpha}\in \mathbb{N}_0^d}\Theta_{\boldsymbol{\alpha}}^2$.

We are ready to prove Proposition~\ref{prop:f-Expand}.
\begin{proof}[Proof of Proposition~\ref{prop:f-Expand}] In what follows, all expectations are taken w.r.t.\,$\bX\sim \mathcal{D}^{\otimes d}$ unless stated otherwise. Let $r=\mathrm{rank}_S(\bT)$. In light of the definition of symmetric tensor rank~\eqref{eq:SymRank}, we obtain that there exists $\lambda_1,\dots,\lambda_r\in\mathbb{R}$ and $\boldsymbol{v}_1,\dots,\boldsymbol{v}_r\in\R^d$ for which
\begin{equation}\label{eq:F_TEN}
    \bT = \sum_{1\le i\le r}\lambda_i \boldsymbol{v}_i^{\otimes \ell}.
\end{equation}
For $\boldsymbol{\alpha}=(\alpha_1,\dots,\alpha_d)\in\mathbb{N}_0^d$, we use Proposition~\ref{prop:BASIS} to expand
\begin{equation}\label{eq:Theta-Coeff}
  \langle \bT,\bX^{\otimes \ell}\rangle = \sum_{\boldsymbol{\alpha}\in\N_0^d} \Theta_{\boldsymbol{\alpha}} \boldsymbol{P}_{\boldsymbol{\alpha}}(\bX) \quad\text{where}\quad
 \Theta_{\boldsymbol{\alpha}} = \mathbb{E}\Bigl[\langle \bT,\bX^{\otimes \ell}\rangle\boldsymbol{P}_{\boldsymbol{\alpha}}(\bX)\Bigr]. 
    \end{equation}
    Using the fact $\boldsymbol{P}_{\boldsymbol{\alpha}}$ are orthonormal per Proposition~\ref{Prop:Orthonormal}, we obtain
\begin{equation}\label{eq:LOW-BD}
       \mathbb{E}\Bigl[ \langle \bT,\bX^{\otimes \ell}\rangle^2\Bigr] = \sum_{\boldsymbol{\alpha}\in\N_0^d} \Theta_{\boldsymbol{\alpha}}^2 \ge \sum_{\substack{\boldsymbol{\alpha}\in\N_0^d:\ip{\boldsymbol{1}_d}{\boldsymbol{\alpha}}=\ell}} \Theta_{\boldsymbol{\alpha}}^2,
    \end{equation}
    where $\boldsymbol{1}_d=(1,\dots,1)\in\R^d$. In the remainder of the proof, we 
     calculate $\Theta_{\boldsymbol{\alpha}}$ for $\boldsymbol{\alpha}\in\N_0^d$ with $\ip{\boldsymbol{1}_d}{\boldsymbol{\alpha}}=\ell$.  Fix $\mathcal{I}=(i_1,\dots,i_\ell)\in[d]^\ell$ and let
\begin{equation}\label{eq:T-I-X-I}
   \bT_{\mathcal{I}} := \sum_{1\le i\le r} \lambda_i \boldsymbol{v}_i(i_1)\cdots \boldsymbol{v}_i(i_\ell) \quad\text{and}\quad \bX_{\mathcal{I}}:= \bX(i_1)\cdots \bX(i_\ell),
\end{equation}
so that~\eqref{eq:F_TEN} yields 
\begin{equation}\label{eq:Tensor-Rep2}
    \langle \bT,\bX^{\otimes \ell}\rangle = \sum_{\mathcal{I}\in[d]^\ell}\bT_{\mathcal{I}}\bX_{\mathcal{I}}.
\end{equation}
Next, for any $\mathcal{I}\in[d]^\ell$, define 
\begin{equation}\label{eq:beta-of-I}
    \boldsymbol{\beta}_{\mathcal{I}} = (\beta_1,\dots,\beta_d)\in\N_0^d\quad\text{where}\quad \beta_j = |1\le k\le \ell:i_k = j|.
\end{equation}
The dependence of $\beta_j$ on $\mathcal{I}$ is suppressed for simplicity. Observe that we have $\ip{\boldsymbol{1}_d}{\boldsymbol{\beta}_{\mathcal{I}}}=\ell$ for all $ 
 \mathcal{I}\in[d]^\ell$. Moreover, for any $\mathcal{I},\mathcal{I}'$ with $\boldsymbol{\beta}_{\mathcal{I}}=\boldsymbol{\beta}_{\mathcal{I}'}$, 
 \begin{equation}\label{eq:T_i-eq-T-i-pr}
     \bT_{\mathcal{I}} = \sum_{1\le j\le r}\lambda_j \boldsymbol{v}_j(1)^{\beta_1}\cdots \boldsymbol{v}_j(d)^{\beta_d} =\bT_{\mathcal{I}'}.
  \end{equation}
 \begin{lemma}\label{lemma:Key-Ort-Poly}
     Given $\mathcal{I}\in[d]^\ell$ and an $\boldsymbol{\alpha}\in\N_0^d$ with $\ip{\boldsymbol{1}_d}{\boldsymbol{\alpha}}=\ell$, we have 
     \[
\mathbb{E}\bigl[\bX_{\mathcal{I}} \boldsymbol{P}_{\boldsymbol{\alpha}}(\bX)\bigr] =\ind\bigl\{\boldsymbol{\beta}_{\mathcal{I}}=\boldsymbol{\alpha}\bigr\}\cdot \prod_{1\le j\le d}\sqrt{\frac{D_{\alpha_j}}{D_{\alpha_j-1}}},
     \]
     where $\boldsymbol{\beta}_{\mathcal{I}}$ is defined in~\eqref{eq:beta-of-I}. Here, $D_{\alpha_j -1}:=1$ if $\alpha_j=0$.
 \end{lemma}
 \begin{proof}[Proof of Lemma~\ref{lemma:Key-Ort-Poly}] 
For $\boldsymbol{\beta}_{\mathcal{I}}=(\beta_1,\dots,\beta_d)$, we have $\bX_{\mathcal{I}} = \bX(1)^{\beta_1}\cdots \bX(d)^{\beta_d}$. Using the independence of $\bX(i),1\le i\le d$, we have:
 \begin{equation}\label{eq:X_i-times-H-alpha}
\mathbb{E}\bigl[\bX_{\mathcal{I}} \boldsymbol{P}_{\boldsymbol{\alpha}}(\bX)\bigr] = \prod_{1\le j\le d}\mathbb{E}_{X\sim \mathcal{D}}\left[X^{\beta_j}P_{\alpha_j}(X)\right].
  \end{equation}
 Suppose first that $\boldsymbol{\beta}_{\mathcal{I}}\ne \boldsymbol{\alpha}$. Since $\ip{\boldsymbol{1}_d}{\boldsymbol{\beta}_{\mathcal{I}}}=\ell$, there exists a $j$ such that $\beta_j<\alpha_j$. Then, Part $\mathrm{(b)}$ of Proposition~\ref{Prop:Orthonormal} implies that~\eqref{eq:X_i-times-H-alpha} evaluates to zero. Suppose next that $\boldsymbol{\beta}_{\mathcal{I}} = \boldsymbol{\alpha}$. Using Proposition~\ref{Prop:Orthonormal} again, we obtain
 \[
 \mathbb{E}_{X\sim \mathcal{D}}\left[X^{\beta_j}P_{\alpha_j}(X)\right] = \sqrt{\frac{D_{\alpha_j}}{D_{\alpha_j -1}}}
 \]
 Combining the last display with~\eqref{eq:X_i-times-H-alpha} establishes Lemma~\ref{lemma:Key-Ort-Poly}.
 \end{proof}
 We are ready to calculate $\Theta_{\boldsymbol{\alpha}}$. Given any $\boldsymbol{\alpha} \in \N_0^d$ with $\ip{\boldsymbol{1}_d}{\boldsymbol{\alpha}}=\ell$, there exists
\begin{equation}\label{eq:COUNT}
     N_{\boldsymbol{\alpha}}:=\frac{\ell!}{\alpha_1 ! \alpha_2 ! \cdots \alpha_d!}
 \end{equation}
values of $\mathcal{I}\in[d]^\ell$ such that $\boldsymbol{\beta}_{\mathcal{I}}=\boldsymbol{\alpha}$.   Using~\eqref{eq:T_i-eq-T-i-pr}, denote by $\bT_{\boldsymbol{\alpha}}$  the common value attained by $\bT_{\mathcal{I}}$ for any $\mathcal{I}$ with $\boldsymbol{\beta}_{\mathcal{I}}=\boldsymbol{\alpha}$. We have
\begin{align}
\Theta_{\boldsymbol{\alpha}}&=\mathbb{E}\left[\boldsymbol{P}_{\boldsymbol{\alpha}}(\bX)\sum_{\mathcal{I}\in[d]^\ell}\bT_{\mathcal{I}}\bX_{\mathcal{I}}\right]  \label{eq:USE-EXP-and-TENS}\\  &= \sum_{\substack{\mathcal{I}\in[d]^\ell \\ \boldsymbol{\beta}_{\mathcal{I}} = \boldsymbol{\alpha}}} \left(\prod_{1\le j\le d}\sqrt{\frac{D_{\alpha_j}}{D_{\alpha_j-1}}}\right)\bT_{\mathcal{I}} \label{eq:USE-LEMMA}\\
&=N_{\boldsymbol{\alpha}}\left(\prod_{1\le j\le d}\sqrt{\frac{D_{\alpha_j}}{D_{\alpha_j-1}}}\right)\bT_{\boldsymbol{\alpha}}\label{eq:USE-N-alpha},
\end{align}
where~\eqref{eq:USE-EXP-and-TENS} follows by combining~\eqref{eq:Theta-Coeff} and~\eqref{eq:Tensor-Rep2}, \eqref{eq:USE-LEMMA} follows from Lemma~\ref{lemma:Key-Ort-Poly}, and~\eqref{eq:USE-N-alpha} follows from~\eqref{eq:COUNT}. Combining~\eqref{eq:LOW-BD} and~\eqref{eq:USE-N-alpha} then yields 
\begin{align}
    \mathbb{E}\Bigl[\langle \bT,\bX^{\otimes \ell}\rangle^2\Bigr] &\ge \sum_{\substack{\boldsymbol{\alpha}\in\N_0^d:\ip{\boldsymbol{1}_d}{\boldsymbol{\alpha}}=\ell}} N_{\boldsymbol{\alpha}}^2 \left(\prod_{1\le j\le d}\frac{D_{\alpha_j}}{D_{\alpha_j-1}}\right)
\bT_{\boldsymbol{\alpha}}^2 \nonumber \\
&\ge \Xi \sum_{\substack{\boldsymbol{\alpha}\in\N_0^d:\ip{\boldsymbol{1}_d}{\boldsymbol{\alpha}}=\ell}} N_{\boldsymbol{\alpha}}^2 
\bT_{\boldsymbol{\alpha}}^2, \label{eq:LOW-BD-2}
\end{align}
where 
\begin{equation}\label{eq:XI-CONST}
    \Xi = \mathcal{C}(\mathcal{D},\ell)^d\quad\text{and}\quad \mathcal{C}(\mathcal{D},\ell) = \frac{\min_{0\le i\le \ell}D_i}{\max\{\max_{0\le i\le \ell} D_i,1\}}.
\end{equation}
Part $\mathrm{(a)}$ of Proposition~\ref{Prop:Orthonormal} ensures that $0<D_i<\infty$ for all $i$. Consequently, $\mathcal{C}(\mathcal{D},\ell)>0$ is a finite constant depending solely on $\ell$ and the moments $m_i$ of $\mathcal{D}$.

At the same time, 
\[
\|\bT\|_F^2 = \sum_{\substack{\boldsymbol{\alpha}\in\N_0^d:\ip{\boldsymbol{1}_d}{\boldsymbol{\alpha}}=\ell}} N_{\boldsymbol{\alpha}} T_{\boldsymbol{\alpha}}^2.
\]
Combining the last display with~\eqref{eq:LOW-BD-2}, we obtain
\[
\mathbb{E}\bigl[\langle \bT,\bX^{\otimes \ell}\rangle^2\bigr] \ge \Xi \cdot \|\bT\|_F^2,
\]
establishing Proposition~\ref{prop:f-Expand}.
\end{proof}

\section{Packing Numbers for Tensors: Proof of Proposition~\ref{prop:TENSOR_PACKING}}\label{sec:Pf-Packing}
The following result is essentially a variant of the Gilbert-Varshamov bound from coding theory~\cite{gilbert1952comparison,varshamov1957estimate}. 
\begin{lemma}\label{lemma:GV}
There exists $d_0$ and a universal constant $c>0$ such that the following holds. Fix any $d\ge d_0$ and $\epsilon>0$. Then, there exists a set $\{v_1',\dots,v_N'\}\subset \{\pm 1\}^d$ such that
    \begin{itemize}
        \item $N=e^{cd\epsilon^2}$.
        \item For every $1\le i<j\le N$, $\frac1d\bigl|\ip{v_i'}{v_j'}\bigr|\le \epsilon$.
    \end{itemize}
\end{lemma}
\begin{proof}[Proof of Lemma~\ref{lemma:GV}]
  Our proof is based on the probabilistic method~\cite{alon2016probabilistic}. Suppose $v_i'(j)$, $i\in[N]$ and $j\in[d]$ are i.i.d.\,Rademacher variables:
  \[
  \mathbb{P}\bigl[v_i'(j)=1\bigr] = \mathbb{P}\bigl[v_i'(j)=-1\bigr] = \frac12.
  \]
  Using Bernstein's inequality, it holds that for some universal constant $C>0$,
\[
\mathbb{P}\left[\frac1d\left|\ip{v_i'}{v_j'}\right|> \epsilon\right]\le \exp\left(-Cd\epsilon^2\right).
\]
Suppose $N = e^{cd\epsilon^2}$, where $c<C/2$. Taking a union bound over $1\le i<j\le N$, we obtain
\[
\mathbb{P}\left[\exists 1\le i<j\le N:\frac1d\left|\ip{v_i'}{v_j'}\right|>\epsilon\right]\le \binom{N}{2}e^{-Cd\epsilon^2} \le \exp\left((2c-C)d\epsilon^2\right).
\]
As $2c-C<0$, we thus have
\[
\mathbb{P}\left[\frac1d\left|\ip{v_i'}{v_j'}\right|\le \epsilon,\forall 1\le i<j\le N\right]>0,
\]
establishing Lemma~\ref{lemma:GV}.
\end{proof}
Equipped with Lemma~\ref{lemma:GV}, we are ready to establish Proposition~\ref{prop:TENSOR_PACKING}.

\begin{proof}[Proof of Proposition~\ref{prop:TENSOR_PACKING}]
We first apply Lemma~\ref{lemma:GV} with $\epsilon = r^{-0.01}$ to obtain the set $\{v_1',\dots,v_N'\}$, where $N=\exp(cdr^{-0.02})$. Next, let $v_i = v_i'/\sqrt{d}$, $1\le i\le N$,  and consider
\[
\Psi = \left\{\bT_S = \sum_{i\in S}v_i^{\otimes \ell} : S\subset [N],|S|=r\right\}.
\]
We prove that $\Psi$ satisfies the desired conditions. 
\paragraph{Cardinality Bound} Since $r=o(d^{50}/\log^{50} d)$, we have
\[
r^{\frac{1}{50}}\log r = o\left(\frac{d}{\log d}\right)\cdot O(\log d) = o(d).
\]
Consequently, $r\log r = o(dr^{0.98})$. With this, we obtain
\begin{align*}
    |\Psi| =  \binom{\exp(cdr^{-0.02})}{r}&\ge \left(\frac{\exp\bigl(cdr^{-0.02}\bigr)}{r}\right)^r = \exp\left(cdr^{0.98} - r\log r\right)=e^{\Omega(dr^{0.98})}.
\end{align*}
\paragraph{Frobenius Norm Estimate} Fix distinct $S,S'\subset[N]$ with $|S|=|S'|=r$ and suppose that $|S\cap S'|=r-k$ for some $1\le k\le r$ (note that $k\ge 1$ as $S\ne S'$). Let $S\setminus S' =\{q_1,\dots,q_k\}$ and $S'\setminus S = \{r_1,\dots,r_k\}$ where $\{q_1,\dots,q_k,r_1,\dots,r_k\}$ consists of $2k$ distinct vectors and $q_i,r_i\in \{v_1,\dots,v_N\}$. We write
\[
\bT_S - \bT_{S'} = \sum_{1\le i\le k}q_i^{\otimes \ell} - \sum_{1\le i\le k}r_i^{\otimes \ell}.
\]
With this, we have 
\begin{align}
    &\|\bT_S - \bT_{S'}\|_F^2 \nonumber\\ 
    &= \sum_{i_1,\dots,i_\ell \in[d]} \left(\sum_{1\le i\le k}q_i(i_1)\cdots q_i(i_\ell) - \sum_{1\le i\le k} r_i(i_1)\cdots r_i(i_\ell)\right)^2 \nonumber\\
    &=\sum_{i_1,\dots,i_\ell \in[d]} \left[\left(\sum_{1\le i\le k}q_i(i_1)\cdots q_i(i_\ell)\right)^2 - 2\left(\sum_{1\le i\le k}q_i(i_1)\cdots q_i(i_\ell)\right)\left(\sum_{1\le i\le k}r_i(i_1)\cdots r_i(i_\ell)\right) \right. \nonumber \\
    &+ \left. \left(\sum_{1\le i\le k}r_i(i_1)\cdots r_i(i_\ell)\right)^2\right].\label{eq:TFB}
\end{align}
We next calculate each term in~\eqref{eq:TFB}. First, we have
\begin{align}
&\sum_{i_1,\dots,i_\ell\in[d]} \left(\sum_{1\le i\le k}q_i(i_1)\cdots q_i(i_\ell)\right)^2  \nonumber \\
    &=\sum_{i_1,\dots,i_\ell \in[d]} \sum_{1\le i\le k}q_i(i_1)^2 \cdots q_i(i_\ell)^2 + 2\sum_{i_1,\dots,i_\ell \in[d]}\sum_{i\ne i'}q_i(i_1)q_{i'}(i_1)\cdots q_i(i_\ell) q_{i'}(i_\ell) \nonumber\\
    &=\sum_{1\le i\le k}\|q_i\|_2^{2\ell} + \sum_{i\ne i'}\ip{q_i}{q_{i'}}^\ell, \label{eq:TFB-2}
\end{align}
where the last step follows by switching the order of summations. Similarly, 
\begin{equation}\label{eq:TFB-3}
    \sum_{i_1,\dots,i_\ell\in[d]} \left(\sum_{1\le i\le k}r_i(i_1)\cdots r_i(i_k)\right)^2 = \sum_{1\le i\le k}\|r_i\|_2^{2\ell} + \sum_{i\ne i'}\ip{r_i}{r_{i'}}^\ell.
\end{equation}
Moreover,
\begin{align}
   &\sum_{i_1,\dots,i_\ell\in[d]} \left(\sum_{1\le i\le k}q_i(i_1)\cdots q_i(i_\ell)\right)\left(\sum_{1\le i\le k}r_i(i_1)\cdots r_i(i_\ell)\right) \nonumber  \\
&=\sum_{1\le i,j\le k}\sum_{i_1,\dots,i_\ell\in[d]}q_i(i_i)r_j(i_1)\cdots q_i(i_\ell)r_j(i_\ell)\nonumber \\
   &=\sum_{1\le i,j\le k}\ip{q_i}{r_j}^\ell.\label{eq:TFB-4}
\end{align}
We now combine~\eqref{eq:TFB},~\eqref{eq:TFB-2},~\eqref{eq:TFB-3} and~\eqref{eq:TFB-4} to arrive at
\begin{align}
    \|\bT_S-\bT_{S'}\|_F^2 
    &= \sum_{1\le i\le k}\|q_i\|_2^{2\ell} + \sum_{i\ne i'}\ip{q_i}{q_{i'}}^\ell + \sum_{1\le i\le k}\|r_i\|_2^{2\ell}\nonumber + \sum_{i\ne i'}\ip{r_i}{r_{i'}}^{2\ell} +\sum_{1\le i,j\le k}\ip{q_i}{r_j}^\ell\nonumber \\
    &=2k +\sum_{i\ne i'}\ip{q_i}{q_{i'}}^\ell + \sum_{i\ne i'}\ip{r_i}{r_{i'}}^\ell + \sum_{1\le i,j\le k}\ip{q_i}{r_j}^\ell \label{eq:USE-NORMS} \\
    &\ge 2k - 3k^2 r^{-0.01\ell} \label{eq:USE-TR-IN},
\end{align}
where~\eqref{eq:USE-NORMS} follows by recalling that $q_i,r_j\in\{v_1,\dots,v_N\}$ so that $\|q_i\|_2= \|r_j\|_2=1$ and~\eqref{eq:USE-TR-IN} follows from the fact
\[
|\ip{q_i}{q_{i'}}|\le r^{-0.01},\quad |\ip{r_i}{r_{i'}}|\le r^{-0.01},\quad |\ip{q_i}{r_j}|\le r^{-0.01}
\]
for any $1\le i<i'\le k$ and any $1\le i,j\le k$ together with the triangle inequality. Combining~\eqref{eq:USE-TR-IN} together with $\ell\ge 101$ and $k\ge 1$, we obtain
\[
\|\bT_S-\bT_{S'}\|_F^2 = 2k\left(1-\frac32 kr^{-0.01\ell}\right)\ge 2\left(1-\frac32 r^{-\frac{1}{100}}\right)\ge 1,
\]
for all large enough $r$. 

We lastly check the integrality of entries and $\max_{\bT\in \Psi}\|\bT\|_F\le r$. The fact $d^{\ell/2}\bT\in\mathbb{Z}^{d\times \cdots \times d}$ is clear. Moreover, for any $\bT_S\in \Psi$, we have
\[
\|\bT_S\|_F \le \left\|\sum_{i\in S}v_i^{\otimes \ell}\right\|_F\le \sum_{i\in S}\|v_i^{\otimes \ell}\|_F=r,
\]
where the last step follows from the fact $|S|=r$ and $\|v_i^{\otimes \ell}\|_F = \|v_i\|_2^\ell = 1$. This completes the proof of Proposition~\ref{prop:TENSOR_PACKING}.
\end{proof}
\section{Sample-Complexity Lower Bound: Proof of Theorem~\ref{thm:FANO-CONV}}\label{pf:FANO}
\begin{proof}[Proof of Theorem~\ref{thm:FANO-CONV}]
     Suppose that $\boldsymbol{Y}^{(N)}$ denotes $(Y_1,\dots,Y_N)$ and $\bX^{(N)}$  denotes the collection $\bX_1,\dots,\bX_N\in\R^d$, where $Y_i = \ip{\bT^*}{\bX_i^{\otimes \ell}}$. We first notice 
    \begin{align*}
        |Y_i| &= \bigl|\ip{\bT^*}{\bX_i^{\otimes \ell}}\bigr|\le \|\bT^*\|_F \cdot \|\bX_i^{\otimes \ell}\|_F \le r(B\sqrt{d})^\ell,
    \end{align*}
    using Cauchy-Schwarz inequality, the fact $\|\bT^*\|_F\le r$ since $\bT^*\sim \mathrm{Uniform}(\Psi)$ and $\max_{\bT\in \Psi}\|\bT\|_F\le r$ per Proposition~\ref{prop:TENSOR_PACKING}, as well as the fact
    \[
    \|\bX_i^{\otimes \ell}\|_F =\sqrt{\sum_{i_1,\dots,i_\ell\in [d]} \bX_i(i_1)^2 \cdots \bX_i(i_\ell)^2} = \|\bX_i\|_2^\ell \le (B\sqrt{d})^\ell.
    \]
    Furthermore, due to Proposition~\ref{prop:TENSOR_PACKING} and the fact that $\bX_i$ are integer-valued
    \[
d^{\ell/2}\boldsymbol{Y}^{(N)}\in [-r(Bd)^\ell,r(Bd)^\ell]^N \cap \mathbb{Z}^N.
    \]
    Thus, 
\begin{equation}\label{eq:H-of-YN}
H\bigl(\boldsymbol{Y}^{(N)}\bigr) =    H\left(d^{\ell/2}\boldsymbol{Y}^{(N)}\right) \le \log\left(\left(2r(Bd)^\ell+1\right)^N\right) \le N\log\left(3r(Bd)^\ell\right),
    \end{equation}
    since the uniform distribution maximizes the entropy. 
    
    Using Fano's inequality~\cite[Theorem~2.10.1]{thomas2006elements}, we obtain
\begin{equation}\label{eq:FANO-LB}    \mathbb{P}\bigl[\widehat{\bT}\ne \bT^*\bigr]\ge \frac{H\bigl(\bT^*|\boldsymbol{Y}^{(N)},\bX^{(N)}\bigr) -1}{\log |\Psi|}.
    \end{equation}
    Next, the chain rule for conditional entropy yields
\begin{align}
H\bigl(\bT^*,\boldsymbol{Y}^{(N)}|\bX^{(N)}\bigr) &= H\bigl(\bT^*|\boldsymbol{Y}^{(N)},\bX^{(N)}\bigr) + H\bigl(\boldsymbol{Y}^{(N)}| \bX^{(N)}\bigr)\nonumber \\
&\le H\bigl(\bT^*|\boldsymbol{Y}^{(N)},\bX^{(N)}\bigr) + H\bigl(\boldsymbol{Y}^{(N)}\bigr)\label{eq:COND-REDUCE-ENT} \\
&\le H\bigl(\bT^*|\boldsymbol{Y}^{(N)},\bX^{(N)}\bigr) + N\log\left(3r(Bd)^\ell\right)\label{UNIF-BD}
\end{align}
where~\eqref{eq:COND-REDUCE-ENT} is a consequence of the fact conditioning reduces entropy and~\eqref{UNIF-BD} follows from~\eqref{eq:H-of-YN}. Furthermore, another application of chain rule yields
\begin{equation}\label{eq:CHAIN-2}
H\bigl(\bT^*,\boldsymbol{Y}^{(N)}|\bX^{(N)}\bigr)=H\bigl(\boldsymbol{Y}^{(N)}|\bT^*,\bX^{(N)}\bigr) + H\bigl(\bT^*|\bX^{(N)}\bigr) = H\bigl(\bT^*\bigr)=\log|\Psi|,
\end{equation}
using the fact $Y_i$ is a function of $\bT^*$ and $\bX_i$ only, so that $H\bigl(\boldsymbol{Y}^{(N)}|\bT^*,\bX^{(N)}\bigr)=0$ and $H(\bT^* | \bX^{(N)})=H(\bT^*)$ as $\bT^*$ and $\bX^{(N)}$ are independent. Combining~\eqref{eq:FANO-LB},~\eqref{UNIF-BD} and~\eqref{eq:CHAIN-2} and using the fact $\log|\Psi| =\Omega(dr^{0.98})$ so that $\log^{-1}|\Psi|=o_{d,r}(1)$, we obtain
\begin{align*}
\mathbb{P}\bigl[\widehat{\bT}\ne \bT^*\bigr]&\ge 1 - \frac{N\left(\log 3r + \ell \log (Bd)\right)}{\log|\Psi|} - o_{d,r}(1)\\
&\ge 1 - O\left(\frac{N\left(\log r + \ell \log (Bd)\right)}{dr^{0.98}}\right) - o_{d,r}(1) \\
&\ge \delta -o_{d,r}(1)
\end{align*}
if $N=O\left(dr^{0.98}/(\log r + \ell \log (Bd))\right)$. This completes the proof of Theorem~\ref{thm:FANO-CONV}.
\end{proof}
\section{Proof of Theorem~\ref{thm:ERM-Sample}}\label{pf:ERM_SAMPLE}
Denote by $\mathcal{S}_\ell$ the set of all symmetric order-$\ell$ tensors. A standard balls and bins argument shows that $\mathrm{dim}(\mathcal{S}_\ell) = \binom{d+\ell-1}{\ell}=N_{d,\ell}^*$, see e.g.~\cite[Proposition~3.4] {comon2008symmetric}. Fix an arbitrary order on $\mathcal{J}:= \bigl\{\boldsymbol{\alpha}\in\mathbb{N}_0^d:\alpha_1+\cdots+\alpha_d=\ell\bigr\}$ and consider the matrix $\M\in\R^{N\times N_{d,\ell}^*}$ with rows $\M_i$, where
\[
\M_i = \Bigl[\bX_i(1)^{\alpha_1}\cdots \bX_i(d)^{\alpha_d}:\boldsymbol{\alpha}=(\alpha_1,\dots,\alpha_d)\in\mathcal{J}\Bigr]\in\R^{N_{d,\ell}^*}.
\]
\subsection{Part $\mathrm{(a)}$} Suppose that $N\ge N_{d,\ell}^*$ and take any $\bT\in\mathcal{S}_\ell$ such that $\mathcal{L}_S(\bT)=0$. In particular, $\langle \bT,\bX_i^{\otimes \ell} \rangle= \langle \bT^*,\bX_i^{\otimes \ell}\rangle$ for all $i\in[N]$, implying 
\[
\M \bT_{\Delta}=\boldsymbol{0}\in\R^N,\quad\text{where}\quad \bT_{\Delta} = \bT - \bT^*.
\]
We will establish that $\mathrm{ker}(\M)=\{\boldsymbol{0}\}$. Consider $\M' \in\R^{N_{d,\ell}^*\times N_{d,\ell}^*}$ obtained by retaining the first $N_{d,\ell}^*$ rows  of $\M$. It is not hard to see that $|\mathcal{M}'|$ is a polynomial of continuous variables $\bX_i,i\in[N]$ that is not identically zero. This establishes $\mathbb{P}[|\mathcal{M}'|\ne 0]=1$, using standard results, e.g.,~\cite{caron2005zero}.
\subsection{Part $\mathrm{(b)}$}
Notice that since $N<N_{d,\ell}^*$, we have by the rank-nullity theorem~\cite{horn2012matrix} that there exists a $\boldsymbol{0}\ne \bar{\bT}\in\mathcal{S}_\ell$ such that $\M \bar{\bT}= \boldsymbol{0}\in\R^N$. Consider now 
\begin{equation}\label{eq:T-lambda}
    \bT(\lambda) = \bT^* + \lambda \bar{\bT}.
\end{equation}
We have that $\mathcal{L}_S(\bT(\lambda)) = 0,\forall \lambda$. Furthermore, for 
\[
\mathcal{L}'(\bT):= \mathbb{E}_{\bX\sim \mathcal{D}^{\otimes d}}\bigl[\bigl(\langle \bT,\bX^{\otimes \ell}\rangle - \langle \bT^*,\bX^{\otimes \ell}\rangle\bigr)^2\bigr],
\]
we have \[
\mathcal{L}'(\bT(\lambda)) = \lambda^2 \mathbb{E}_{\bX\sim \mathcal{D}^{\otimes d}}\bigl[\langle \bar{\bT},\bX^{\otimes \ell}\rangle^2\bigr].
\]As $\bar{\bT}\ne 0$, we have $\mathbb{E}_{\bX\sim \mathcal{D}^{\otimes d}}\bigl[\langle \bar{\bT},\bX^{\otimes \ell}\rangle^2\bigr]>0$; so $\mathcal{L}'(\bT(\lambda))\ge M$ provided $\lambda$ is sufficiently large.

It thus suffices to show that $\bT(\lambda)$ admits a decomposition 
\[
\bT(\lambda) = \sum_{1\le j\le m}c_j \boldsymbol{v}_j^{\otimes \ell}
\]
for some $m\in\mathbb{N}$, $c_1,\dots,c_m\in\R$ and $\boldsymbol{v}_1,\dots,\boldsymbol{v}_m\in\R^d$. Fix indices $i_1,\dots,i_\ell \in[d]$ and define
\begin{equation}\label{eq:O-dot}
    e_{i_1}\odot e_{i_2}\odot \cdots \odot e_{i_\ell} = \frac{1}{\ell!} \sum_{\pi\in K_\ell}e_{i_{\pi(1)}} \otimes \cdots \otimes e_{i_{\pi(\ell)}},
\end{equation}
where $K_\ell$ is the set of all permutations $\pi:[\ell]\to[\ell]$. The following is a well-known fact:
\begin{proposition}{\cite[Proposition~3.4]{comon2008symmetric}}\label{prop:Basis}
    The set
    \[
    \bigl\{e_{i_1}\otimes \cdots \otimes e_{i_\ell} : 1\le i_1\le \cdots \le i_\ell \le d\bigr\}
    \]
    is a basis for $\mathcal{S}_\ell$.
\end{proposition}
Using Proposition~\ref{prop:Basis}, we thus obtain
\[
\bT(\lambda) = \sum_{1\le i_1\le \cdots \le i_\ell \le d}\bigl(\bT(\lambda)\bigr)_{i_1,\dots,i_\ell} e_{i_1}\odot \cdots \odot e_{i_\ell}.
\]
Using the polarization identity~\cite[Equation~A.4]{thomas2014polarization} we obtain that there exists $\boldsymbol{v}_j^{(i_1,\dots,i_\ell)}\in\R^d$, $1\le j\le 2^\ell$ and signs $\epsilon_j^{(i_1,\dots,i_\ell)}\in\{-1,1\}$ such that
\[
e_{i_1}\odot \cdots \odot e_{i_\ell} = \sum_{1\le j\le 2^\ell}\epsilon_j^{(i_1,\dots,i_\ell)} \left(\boldsymbol{v}_j^{(i_1,\dots,i_\ell)}\right)^{\otimes \ell}.
\]
Combining the last two displays, we have
\[
\bT(\lambda) = \sum_{1\le i_1\le \cdots \le i_\ell \le d} \sum_{1\le j\le 2^\ell}\bigl(\bT(\lambda)\bigr)_{i_1,\dots,i_\ell} \epsilon_j^{(i_1,\dots,i_\ell)} \left(\boldsymbol{v}_j^{(i_1,\dots,i_\ell)}\right)^{\otimes \ell},
\]
establishing that $\mathrm{rank}_S(\bT(\lambda))\le (2d)^\ell$ for all $\lambda$. This completes the proof.
\subsubsection*{Acknowledgments}
I would like to thank anonymous ALT 2025 referees and the area chair for their valuable feedback which helped improve the presentation and strengthen the results.
\bibliographystyle{amsalpha}
\bibliography{erenbiblio}

\newcommand{\etalchar}[1]{$^{#1}$}
\providecommand{\bysame}{\leavevmode\hbox to3em{\hrulefill}\thinspace}
\providecommand{\MR}{\relax\ifhmode\unskip\space\fi MR }
\providecommand{\MRhref}[2]{%
  \href{http://www.ams.org/mathscinet-getitem?mr=#1}{#2}
}
\providecommand{\href}[2]{#2}
\begin{thebibliography}{SMVEZ20}

\bibitem[ABKZ20]{abbara2020universality}
Alia Abbara, Antoine Baker, Florent Krzakala, and Lenka Zdeborov{\'a}, \emph{On the universality of noiseless linear estimation with respect to the measurement matrix}, Journal of Physics A: Mathematical and Theoretical \textbf{53} (2020), no.~16, 164001.

\bibitem[ARB20]{ahmed2020tensor}
Talal Ahmed, Haroon Raja, and Waheed~U Bajwa, \emph{Tensor regression using low-rank and sparse tucker decompositions}, SIAM Journal on Mathematics of Data Science \textbf{2} (2020), no.~4, 944--966.

\bibitem[Arg12]{MSC-Post}
Martin Argerami, \emph{Orthonormal basis for product $l^2$ space}, Mathematics Stack Exchange, 2012, \url{https://math.stackexchange.com/q/105486} (version: 2024-07-30).

\bibitem[AS16]{alon2016probabilistic}
Noga Alon and Joel~H Spencer, \emph{The probabilistic method}, John Wiley \& Sons, 2016.

\bibitem[AXY24]{auddy2024tensor}
Arnab Auddy, Dong Xia, and Ming Yuan, \emph{Tensor methods in high dimensional data analysis: Opportunities and challenges}, arXiv preprint arXiv:2405.18412 (2024).

\bibitem[BB06]{bagnoli2006log}
Mark Bagnoli and Ted Bergstrom, \emph{Log-concave probability and its applications}, Rationality and Equilibrium: A Symposium in Honor of Marcel K. Richter, Springer, 2006, pp.~217--241.

\bibitem[BC15]{bobkov2015concentration}
Sergey~G Bobkov and Gennadiy~P Chistyakov, \emph{On concentration functions of random variables}, Journal of Theoretical Probability \textbf{28} (2015), no.~3, 976--988.

\bibitem[BJKK17]{bhatia2017consistent}
Kush Bhatia, Prateek Jain, Parameswaran Kamalaruban, and Purushottam Kar, \emph{Consistent robust regression}, Advances in Neural Information Processing Systems \textbf{30} (2017).

\bibitem[BKBY18]{basu2018iterative}
Sumanta Basu, Karl Kumbier, James~B Brown, and Bin Yu, \emph{Iterative random forests to discover predictive and stable high-order interactions}, Proceedings of the National Academy of Sciences \textbf{115} (2018), no.~8, 1943--1948.

\bibitem[BMR21]{bartlett2021deep}
Peter~L Bartlett, Andrea Montanari, and Alexander Rakhlin, \emph{Deep learning: a statistical viewpoint}, Acta numerica \textbf{30} (2021), 87--201.

\bibitem[BMV{\etalchar{+}}18]{banks2018information}
Jess Banks, Cristopher Moore, Roman Vershynin, Nicolas Verzelen, and Jiaming Xu, \emph{Information-theoretic bounds and phase transitions in clustering, sparse pca, and submatrix localization}, IEEE Transactions on Information Theory \textbf{64} (2018), no.~7, 4872--4894.

\bibitem[BS05]{beckmann2005tensorial}
Christian~F Beckmann and Stephen~M Smith, \emph{Tensorial extensions of independent component analysis for multisubject fmri analysis}, Neuroimage \textbf{25} (2005), no.~1, 294--311.

\bibitem[BTT13]{bien2013lasso}
Jacob Bien, Jonathan Taylor, and Robert Tibshirani, \emph{A lasso for hierarchical interactions}, Annals of statistics \textbf{41} (2013), no.~3, 1111.

\bibitem[BTY{\etalchar{+}}21]{bi2021tensors}
Xuan Bi, Xiwei Tang, Yubai Yuan, Yanqing Zhang, and Annie Qu, \emph{Tensors in statistics}, Annual review of statistics and its application \textbf{8} (2021), no.~1, 345--368.

\bibitem[CCS10]{cai2010singular}
Jian-Feng Cai, Emmanuel~J Cand{\`e}s, and Zuowei Shen, \emph{A singular value thresholding algorithm for matrix completion}, SIAM Journal on optimization \textbf{20} (2010), no.~4, 1956--1982.

\bibitem[CGLM08]{comon2008symmetric}
Pierre Comon, Gene Golub, Lek-Heng Lim, and Bernard Mourrain, \emph{Symmetric tensors and symmetric tensor rank}, SIAM Journal on Matrix Analysis and Applications \textbf{30} (2008), no.~3, 1254--1279.

\bibitem[CLMW11]{candes2011robust}
Emmanuel~J Cand{\`e}s, Xiaodong Li, Yi~Ma, and John Wright, \emph{Robust principal component analysis?}, Journal of the ACM (JACM) \textbf{58} (2011), no.~3, 1--37.

\bibitem[CLQY20]{comon2020topology}
Pierre Comon, Lek-Heng Lim, Yang Qi, and Ke~Ye, \emph{Topology of tensor ranks}, Advances in Mathematics \textbf{367} (2020), 107128.

\bibitem[CMWX20]{cai2020provable}
Jian-Feng Cai, Lizhang Miao, Yang Wang, and Yin Xian, \emph{Provable near-optimal low-multilinear-rank tensor recovery}, arXiv preprint arXiv:2007.08904 (2020).

\bibitem[CP10]{candes2010matrix}
Emmanuel~J Candes and Yaniv Plan, \emph{Matrix completion with noise}, Proceedings of the IEEE \textbf{98} (2010), no.~6, 925--936.

\bibitem[CR12]{candes2012exact}
Emmanuel Candes and Benjamin Recht, \emph{Exact matrix completion via convex optimization}, Communications of the ACM \textbf{55} (2012), no.~6, 111--119.

\bibitem[CRT06]{candes2006robust}
Emmanuel~J Cand{\`e}s, Justin Romberg, and Terence Tao, \emph{Robust uncertainty principles: Exact signal reconstruction from highly incomplete frequency information}, IEEE Transactions on information theory \textbf{52} (2006), no.~2, 489--509.

\bibitem[CT05a]{candes2005decoding}
Emmanuel~J Candes and Terence Tao, \emph{Decoding by linear programming}, IEEE transactions on information theory \textbf{51} (2005), no.~12, 4203--4215.

\bibitem[CT05b]{caron2005zero}
Richard Caron and Tim Traynor, \emph{The zero set of a polynomial}, WSMR Report (2005), 05--02.

\bibitem[CT06]{thomas2006elements}
Thomas~M Cover and Joy~A Thomas, \emph{Elements of information theory}, Wiley-Interscience, 2006.

\bibitem[CW01]{carbery2001distributional}
Anthony Carbery and James Wright, \emph{Distributional and ${L}^q$ norm inequalities for polynomials over convex bodies in $\mathbb{R}^n$}, Mathematical research letters \textbf{8} (2001), no.~3, 233--248.

\bibitem[DKS17]{diakonikolas2017learning}
Ilias Diakonikolas, Daniel~M Kane, and Alistair Stewart, \emph{Learning multivariate log-concave distributions}, Conference on Learning Theory, PMLR, 2017, pp.~711--727.

\bibitem[DL18]{du2018power}
Simon Du and Jason Lee, \emph{On the power of over-parametrization in neural networks with quadratic activation}, International conference on machine learning, PMLR, 2018, pp.~1329--1338.

\bibitem[dLN{\etalchar{+}}21]{d2021consistent2}
Tommaso d'Orsi, Chih-Hung Liu, Rajai Nasser, Gleb Novikov, David Steurer, and Stefan Tiegel, \emph{Consistent estimation for pca and sparse regression with oblivious outliers}, Advances in Neural Information Processing Systems \textbf{34} (2021), 25427--25438.

\bibitem[dNS21]{d2021consistent1}
Tommaso d’Orsi, Gleb Novikov, and David Steurer, \emph{Consistent regression when oblivious outliers overwhelm}, International Conference on Machine Learning, PMLR, 2021, pp.~2297--2306.

\bibitem[EGKZ20]{emschwiller2020neural}
Matt Emschwiller, David Gamarnik, Eren~C K{\i}z{\i}lda{\u{g}}, and Ilias Zadik, \emph{Neural networks and polynomial regression. demystifying the overparametrization phenomena}, arXiv preprint arXiv:2003.10523 (2020).

\bibitem[ENP12]{eldar2012uniqueness}
Yonina~C Eldar, Deanna Needell, and Yaniv Plan, \emph{Uniqueness conditions for low-rank matrix recovery}, Applied and Computational Harmonic Analysis \textbf{33} (2012), no.~2, 309--314.

\bibitem[GAS{\etalchar{+}}19]{goldt2019dynamics}
Sebastian Goldt, Madhu Advani, Andrew~M Saxe, Florent Krzakala, and Lenka Zdeborov{\'a}, \emph{Dynamics of stochastic gradient descent for two-layer neural networks in the teacher-student setup}, Advances in neural information processing systems \textbf{32} (2019).

\bibitem[Gil52]{gilbert1952comparison}
Edgar~N Gilbert, \emph{A comparison of signalling alphabets}, The Bell system technical journal \textbf{31} (1952), no.~3, 504--522.

\bibitem[GKZ24]{gamarnik2024stationary}
David Gamarnik, Eren~C K{\i}z{\i}lda{\u{g}}, and Ilias Zadik, \emph{Stationary points of a shallow neural network with quadratic activations and the global optimality of the gradient descent algorithm}, Mathematics of Operations Research (2024).

\bibitem[GLM{\etalchar{+}}22]{grotheer2022iterative}
Rachel Grotheer, Shuang Li, Anna Ma, Deanna Needell, and Jing Qin, \emph{Iterative hard thresholding for low cp-rank tensor models}, Linear and Multilinear Algebra \textbf{70} (2022), no.~22, 7452--7468.

\bibitem[GPY18]{ghadermarzy2018learning}
Navid Ghadermarzy, Yaniv Plan, and Ozgur Yilmaz, \emph{Learning tensors from partial binary measurements}, IEEE Transactions on Signal Processing \textbf{67} (2018), no.~1, 29--40.

\bibitem[GRS20]{golowich2020size}
Noah Golowich, Alexander Rakhlin, and Ohad Shamir, \emph{Size-independent sample complexity of neural networks}, Information and Inference: A Journal of the IMA \textbf{9} (2020), no.~2, 473--504.

\bibitem[H{\aa}s89]{haastad1989tensor}
Johan H{\aa}stad, \emph{Tensor rank is np-complete}, Automata, Languages and Programming: 16th International Colloquium Stresa, Italy, July 11--15, 1989 Proceedings 16, Springer, 1989, pp.~451--460.

\bibitem[H{\aa}s90]{haastad1990tensor}
\bysame, \emph{Tensor rank is np-complete}, Journal of algorithms \textbf{11} (1990), no.~4, 644--654.

\bibitem[HJ12]{horn2012matrix}
Roger~A Horn and Charles~R Johnson, \emph{Matrix analysis}, Cambridge University Press, 2012.

\bibitem[HL13]{hillar2013most}
Christopher~J Hillar and Lek-Heng Lim, \emph{Most tensor problems are np-hard}, Journal of the ACM (JACM) \textbf{60} (2013), no.~6, 1--39.

\bibitem[HLC{\etalchar{+}}16]{hung2016detection}
Hung Hung, Yu-Ting Lin, Penweng Chen, Chen-Chien Wang, Su-Yun Huang, and Jung-Ying Tzeng, \emph{Detection of gene--gene interactions using multistage sparse and low-rank regression}, Biometrics \textbf{72} (2016), no.~1, 85--94.

\bibitem[HZC20]{hao2020sparse}
Botao Hao, Anru~R Zhang, and Guang Cheng, \emph{Sparse and low-rank tensor estimation via cubic sketchings}, International conference on artificial intelligence and statistics, PMLR, 2020, pp.~1319--1330.

\bibitem[Jal19]{jalali2019toward}
Shirin Jalali, \emph{Toward theoretically founded learning-based compressed sensing}, IEEE Transactions on Information Theory \textbf{66} (2019), no.~1, 387--400.

\bibitem[JP17]{jalali2017universal}
Shirin Jalali and H~Vincent Poor, \emph{Universal compressed sensing for almost lossless recovery}, IEEE Transactions on Information Theory \textbf{63} (2017), no.~5, 2933--2953.

\bibitem[KB09]{kolda2009tensor}
Tamara~G Kolda and Brett~W Bader, \emph{Tensor decompositions and applications}, SIAM review \textbf{51} (2009), no.~3, 455--500.

\bibitem[K{\i}z22]{kizildag2022algorithms}
Eren~C K{\i}z{\i}lda\u{g}, \emph{Algorithms and algorithmic barriers in high-dimensional statistics and random combinatorial structures}, Ph.D. thesis, Massachusetts Institute of Technology, 2022.

\bibitem[KM05]{klartag2005geometry}
Boaz Klartag and VD~Milman, \emph{Geometry of log-concave functions and measures}, Geometriae Dedicata \textbf{112} (2005), 169--182.

\bibitem[KMO10]{keshavan2010matrix}
Raghunandan~H Keshavan, Andrea Montanari, and Sewoong Oh, \emph{Matrix completion from a few entries}, IEEE transactions on information theory \textbf{56} (2010), no.~6, 2980--2998.

\bibitem[KWB19]{kunisky2019notes}
Dmitriy Kunisky, Alexander~S Wein, and Afonso~S Bandeira, \emph{Notes on computational hardness of hypothesis testing: Predictions using the low-degree likelihood ratio}, ISAAC Congress (International Society for Analysis, its Applications and Computation), Springer, 2019, pp.~1--50.

\bibitem[Lal]{LalleyNotes}
Steven Lalley, \emph{{Orthogonal Polynomials}}, \url{https://galton.uchicago.edu/~lalley/Courses/386/OrthogonalPolynomials.pdf}.

\bibitem[LSSS14]{livni2014computational}
Roi Livni, Shai Shalev-Shwartz, and Ohad Shamir, \emph{On the computational efficiency of training neural networks}, Advances in neural information processing systems \textbf{27} (2014).

\bibitem[LV07]{lovasz2007geometry}
L{\'a}szl{\'o} Lov{\'a}sz and Santosh Vempala, \emph{The geometry of logconcave functions and sampling algorithms}, Random Structures \& Algorithms \textbf{30} (2007), no.~3, 307--358.

\bibitem[LZ23]{luo2023low}
Yuetian Luo and Anru~R Zhang, \emph{Low-rank tensor estimation via riemannian gauss-newton: Statistical optimality and second-order convergence}, The Journal of Machine Learning Research \textbf{24} (2023), no.~1, 18274--18321.

\bibitem[MBB24]{martin2024impact}
Simon Martin, Francis Bach, and Giulio Biroli, \emph{On the impact of overparameterization on the training of a shallow neural network in high dimensions}, International Conference on Artificial Intelligence and Statistics, PMLR, 2024, pp.~3655--3663.

\bibitem[MHWG14]{mu2014square}
Cun Mu, Bo~Huang, John Wright, and Donald Goldfarb, \emph{Square deal: Lower bounds and improved relaxations for tensor recovery}, International conference on machine learning, PMLR, 2014, pp.~73--81.

\bibitem[MSS06]{mesgarani2006content}
Nima Mesgarani, Malcolm Slaney, and SA~Shamma, \emph{Content-based audio classification based on multiscale spectro-temporal features}, IEEE Transactions on Speech and Audio processing \textbf{14} (2006), no.~3, 920--930.

\bibitem[NS10]{nion2010tensor}
Dimitri Nion and Nicholas~D Sidiropoulos, \emph{Tensor algebra and multidimensional harmonic retrieval in signal processing for mimo radar}, IEEE Transactions on Signal Processing \textbf{58} (2010), no.~11, 5693--5705.

\bibitem[NTS15]{neyshabur2015norm}
Behnam Neyshabur, Ryota Tomioka, and Nathan Srebro, \emph{Norm-based capacity control in neural networks}, Conference on learning theory, PMLR, 2015, pp.~1376--1401.

\bibitem[PJL24]{pensia2024robust}
Ankit Pensia, Varun Jog, and Po-Ling Loh, \emph{Robust regression with covariate filtering: Heavy tails and adversarial contamination}, Journal of the American Statistical Association (2024), 1--12.

\bibitem[RSB15]{riegler2015information}
Erwin Riegler, David Stotz, and Helmut B{\"o}lcskei, \emph{Information-theoretic limits of matrix completion}, 2015 IEEE International Symposium on Information Theory (ISIT), IEEE, 2015, pp.~1836--1840.

\bibitem[RSS17]{rauhut2017low}
Holger Rauhut, Reinhold Schneider, and {\v{Z}}eljka Stojanac, \emph{Low rank tensor recovery via iterative hard thresholding}, Linear Algebra and its Applications \textbf{523} (2017), 220--262.

\bibitem[Sam18]{samworth}
Richard~J. Samworth, \emph{{Recent Progress in Log-Concave Density Estimation}}, Statistical Science \textbf{33} (2018), no.~4, 493 -- 509.

\bibitem[SBRJ19]{suggala2019adaptive}
Arun~Sai Suggala, Kush Bhatia, Pradeep Ravikumar, and Prateek Jain, \emph{Adaptive hard thresholding for near-optimal consistent robust regression}, Conference on Learning Theory, PMLR, 2019, pp.~2892--2897.

\bibitem[SJL18]{soltanolkotabi2018theoretical}
Mahdi Soltanolkotabi, Adel Javanmard, and Jason~D Lee, \emph{Theoretical insights into the optimization landscape of over-parameterized shallow neural networks}, IEEE Transactions on Information Theory \textbf{65} (2018), no.~2, 742--769.

\bibitem[SK12]{sidiropoulos2012multi}
Nicholas~D Sidiropoulos and Anastasios Kyrillidis, \emph{Multi-way compressed sensing for sparse low-rank tensors}, IEEE Signal Processing Letters \textbf{19} (2012), no.~11, 757--760.

\bibitem[SMVEZ20]{sarao2020optimization}
Stefano Sarao~Mannelli, Eric Vanden-Eijnden, and Lenka Zdeborov{\'a}, \emph{Optimization and generalization of shallow neural networks with quadratic activation functions}, Advances in Neural Information Processing Systems \textbf{33} (2020), 13445--13455.

\bibitem[SRT15]{shah2015optimal}
Parikshit Shah, Nikhil Rao, and Gongguo Tang, \emph{Optimal low-rank tensor recovery from separable measurements: Four contractions suffice}, arXiv preprint arXiv:1505.04085 (2015).

\bibitem[Sta89]{stanley1989log}
Richard~P Stanley, \emph{Log-concave and unimodal sequences in algebra, combinatorics, and geometry}, Ann. New York Acad. Sci \textbf{576} (1989), no.~1, 500--535.

\bibitem[Sze39]{szego1939orthogonal}
Gabor Szeg{\"o}, \emph{Orthogonal polynomials}, vol.~23, American Mathematical Soc., 1939.

\bibitem[TBD11]{tan2011rank}
Vincent~YF Tan, Laura Balzano, and Stark~C Draper, \emph{Rank minimization over finite fields: Fundamental limits and coding-theoretic interpretations}, IEEE transactions on information theory \textbf{58} (2011), no.~4, 2018--2039.

\bibitem[Tho14]{thomas2014polarization}
Erik~GF Thomas, \emph{A polarization identity for multilinear maps}, Indagationes Mathematicae \textbf{25} (2014), no.~3, 468--474.

\bibitem[Var57]{varshamov1957estimate}
Rom~Rubenovich Varshamov, \emph{Estimate of the number of signals in error correcting codes}, Docklady Akad. Nauk, SSSR \textbf{117} (1957), 739--741.

\bibitem[VBB19]{venturi2019spurious}
Luca Venturi, Afonso~S Bandeira, and Joan Bruna, \emph{Spurious valleys in one-hidden-layer neural network optimization landscapes}, Journal of Machine Learning Research \textbf{20} (2019), no.~133, 1--34.

\bibitem[Ver10]{vershynin2010introduction}
Roman Vershynin, \emph{Introduction to the non-asymptotic analysis of random matrices}, arXiv preprint arXiv:1011.3027 (2010).

\bibitem[Ver18]{vershynin2018high}
\bysame, \emph{High-dimensional probability: An introduction with applications in data science}, vol.~47, Cambridge university press, 2018.

\bibitem[VSS22]{vardi2022sample}
Gal Vardi, Ohad Shamir, and Nati Srebro, \emph{The sample complexity of one-hidden-layer neural networks}, Advances in Neural Information Processing Systems \textbf{35} (2022), 9139--9150.

\bibitem[Wai09]{wainwright2009information}
Martin~J Wainwright, \emph{Information-theoretic limits on sparsity recovery in the high-dimensional and noisy setting}, IEEE transactions on information theory \textbf{55} (2009), no.~12, 5728--5741.

\bibitem[Wal09]{walther2009inference}
Guenther Walther, \emph{{Inference and modeling with log-concave distributions}}, Statistical Science (2009), 319--327.

\bibitem[WV10]{wu2010renyi}
Yihong Wu and Sergio Verd{\'u}, \emph{R{\'e}nyi information dimension: Fundamental limits of almost lossless analog compression}, IEEE Transactions on Information Theory \textbf{56} (2010), no.~8, 3721--3748.

\bibitem[WWR10]{wang2010information}
Wei Wang, Martin~J Wainwright, and Kannan Ramchandran, \emph{Information-theoretic limits on sparse signal recovery: Dense versus sparse measurement matrices}, IEEE Transactions on Information Theory \textbf{56} (2010), no.~6, 2967--2979.

\bibitem[Xu18]{xu2018minimal}
Zhiqiang Xu, \emph{The minimal measurement number for low-rank matrix recovery}, Applied and Computational Harmonic Analysis \textbf{44} (2018), no.~2, 497--508.

\bibitem[YZ17]{yuan2017incoherent}
Ming Yuan and Cun-Hui Zhang, \emph{Incoherent tensor norms and their applications in higher order tensor completion}, IEEE Transactions on Information Theory \textbf{63} (2017), no.~10, 6753--6766.

\bibitem[ZAZD19]{zhang2019tensor}
Zhengwu Zhang, Genevera~I Allen, Hongtu Zhu, and David Dunson, \emph{Tensor network factorizations: Relationships between brain structural connectomes and traits}, Neuroimage \textbf{197} (2019), 330--343.

\bibitem[ZDJW13]{zhang2013information}
Yuchen Zhang, John Duchi, Michael~I Jordan, and Martin~J Wainwright, \emph{Information-theoretic lower bounds for distributed statistical estimation with communication constraints}, Advances in Neural Information Processing Systems \textbf{26} (2013).

\bibitem[ZK23]{zhang2023covering}
Yifan Zhang and Joe Kileel, \emph{Covering number of real algebraic varieties and beyond: Improved bounds and applications}, arXiv e-prints (2023), arXiv--2311.

\bibitem[ZLZ13]{zhou2013tensor}
Hua Zhou, Lexin Li, and Hongtu Zhu, \emph{Tensor regression with applications in neuroimaging data analysis}, Journal of the American Statistical Association \textbf{108} (2013), no.~502, 540--552.

\end{thebibliography}
\end{document}